\pdfoutput=1 
\documentclass[twoside,11pt]{article}

\usepackage{amsmath, mathtools, amssymb, amsthm}
\usepackage{mathrsfs}
\usepackage{natbib}
\usepackage{microtype}
\usepackage{enumerate, multicol}
\usepackage{bm}
\usepackage{bbm}
\usepackage[linesnumbered,ruled,vlined]{algorithm2e}
\usepackage[font=small,labelfont=bf]{caption}
\usepackage{array} 
\usepackage[table]{xcolor}
\usepackage{tikz}
\usetikzlibrary{positioning, arrows, shapes.geometric, calc}
\usepackage{fullpage}
\usepackage[colorlinks=true, linkcolor=blue, urlcolor=blue, citecolor=blue]{hyperref}
\usepackage[linesnumbered,ruled,vlined]{algorithm2e}
\usepackage{breakcites}


\definecolor{gaussianfill}{RGB}{220,230,250}
\definecolor{kssgfill}{RGB}{220,250,230}
\definecolor{usgfill}{RGB}{245,230,250}

\SetKwInput{KwInput}{Input}                
\SetKwInput{KwOutput}{Output}              
\SetKwRepeat{KwRepeat}{repeat}{until}
\SetKw{Break}{break}
\SetKw{Continue}{continue}
\SetKw{Import}{import}

\theoremstyle{plain}
\newtheorem{theorem}{Theorem}[section]
\newtheorem{lemma}[theorem]{Lemma}
\newtheorem{corollary}[theorem]{Corollary}
\newtheorem{proposition}[theorem]{Proposition}
\newtheorem*{theorem*}{Theorem}
\newtheorem*{corollary*}{Corollary}
\newtheorem*{lemma*}{Lemma}
\newtheorem*{proposition*}{Proposition}

\theoremstyle{definition}
\newtheorem{definition}{Definition}[section]
\newtheorem*{definition*}{Definition}

\newtheorem*{remark*}{Remark}
\newtheorem{example}{Example}
\newtheorem*{example*}{Example}

\newtheorem*{fact*}{Fact}

\numberwithin{equation}{section}

\def\T{^\top}

\newcommand{\npar}{\vspace{0.5em}\par }

\allowdisplaybreaks

\begin{document}

\begin{center}
    {\LARGE Efficient Robust Constrained Signal Detection via Kolmogorov Width Approximations}

    {\large
        \begin{center}
            Yikun Li and Matey Neykov
        \end{center}}

    {Department of Statistics and Data Science\\ Northwestern University\\Evanston, IL 60208\\\texttt{yikunli2028@u.northwestern.edu}~~~~~~~~\texttt{mneykov@northwestern.edu}}
\end{center}

\begin{abstract}
Robust statistical inference often faces a severe computational-statistical gap when dealing with complex parameter spaces. We investigate minimax signal detection in the Gaussian sequence model under strong $\epsilon$-contamination, where the signal belongs to a general prior constraint $K$. Existing optimal tests require computing the exact Kolmogorov $k$-width of $K$, a computationally intractable task for general non-trivial sets. We bridge this gap by proposing a polynomial-time testing framework that universally applies to balanced, type-2, and exactly 2-convex constraints. By leveraging a semidefinite programming relaxation and a modified ellipsoid method equipped with an approximate subgradient oracle, we efficiently approximate the Kolmogorov widths. Remarkably, our unconditional efficient algorithm achieves a robust detection boundary that matches existing upper bounds up to a mere polylogarithmic factor. This establishes a computationally tractable testing solution for a broad class of structured signals without requiring prior knowledge of their exact geometric complexity.
\end{abstract}

\tableofcontents

\section{Introduction}\label{section: introduction}

\subsection*{Robust Statistics}

A central challenge in the application of modern statistical estimation and hypothesis testing lies in specifying the true underlying data distribution. Traditionally, statistical frameworks operate under the idealized assumption that observations are sampled from a well-behaved parametric family, such as the Gaussian distribution. However, this assumption is frequently violated in real-world applications. These violations typically arise from experimental measurement inaccuracies (e.g., \cite{leek2010tackling}), heavy-tailed phenomena inherent in quantitative finance (e.g., \cite{Cont01022001}), or deterministic adversarial outliers, such as data poisoning in machine learning (e.g., \cite{10.5555/3294996.3295110, pmlr-v80-yin18a}). In the presence of such anomalies, even a small fraction of corrupted data can catastrophically compromise standard inference procedures. To illustrate, consider the fundamental task of estimating the population mean from a sample $x _{1},x _{2},\dots,x _{n}$ drawn from $\mathcal{N}(\mu ,1)$. While the sufficient statistic $\bar{x}=\frac{1}{n}\sum\limits_{i=1}^{n}x _{i}$ is minimax optimal in both the estimation problem (under the $\ell _{2}$ norm loss) and the hypothesis testing problem (for prescribed Type I and Type II error), an adversarial perturbation to just a single observation in $\left\{x _{i}\right\}_{i=1}^{n}$ can arbitrarily skew $\bar{x}$ (e.g., forcing it to zero), effectively nullifying its reliability. Conversely, robust alternatives such as the sample median of $\left\{x _{i}\right\}_{i=1}^{n}$ exhibit significantly greater resilience against such data modifications.

Recognizing such vulnerabilities, statisticians have long emphasized the critical role of robustness in methodological development. One of the foundational works formalizing this field was presented by \cite{10.1214/aoms/1177703732}, who introduced the \textit{Huber contamination model} (also known as the \textit{$\epsilon$-contamination model}). In this model, rather than observing samples exclusively from a presumed nominal distribution $\mathbb{P}$, one assumes the data generation process is corrupted by an unknown, adversarially chosen distribution $\mathbb{Q}$. Consequently, the observations are drawn from the mixture distribution $(1-\epsilon)\mathbb{P}+\epsilon\mathbb{Q}$ for some contamination proportion $\epsilon \in (0, 1)$. Crucially, the adversary's choice of $\mathbb{Q}$ is permitted to depend on $\mathbb{P}$. To achieve robust estimation under this framework, Huber also introduced the Huber loss—a piecewise function combining quadratic and linear penalties —-- as a resilient estimation criterion. An even more formidable adversary operates within the \textit{strong $\epsilon $-contamination model}. Under this framework, for a given contamination fraction $\epsilon \in (0,1)$, the adversary is permitted to inspect the original uncorrupted samples and replace up to an $\epsilon $ fraction of them with arbitrary malicious values. This replacement process is highly adaptive, as the adversary possesses full knowledge of the original sample realization, the underlying data-generating mechanism, and even the downstream inference procedures to be deployed. This full adaptivity renders robust inference significantly more challenging than under the standard Huber contamination model. The strong $\epsilon$-contamination model serves as the primary focus of this work (see Definition \ref{def_epsilon_contamination_model} for a formal formulation). Also see \cite{Diakonikolas2019RecentAI} for a comprehensive overview of various contamination models. Alternatively, even in the absence of an explicit adversary, the true underlying distribution may inherently deviate from standard parametric assumptions. A prominent example is the heavy-tailed distribution, where the tail probability decays polynomially rather than exponentially (as is characteristic of Gaussian or exponential distributions). In such scenarios, one typically relies on milder moment assumptions. A substantial body of literature has been developed within this framework. For instance, \cite{10.1214/11-AIHP454} demonstrates that it is possible to recover sub-Gaussian concentration bounds for heavy-tailed distributions with only finite variance. Furthermore, \cite{dd78bb05-37b9-3dee-a39a-31654511afe1} extend this approach to high-dimensional heavy-tailed settings. For related work on the estimation of covariance matrices under similar conditions, see \cite{minsker2018subgaussianestimatorsmeanrandom}. Below we introduce the concepts of Kolmogorov $k$-width and the ellipsoid method algorithm, which play central role in the derivation of the approximation framework in this work.

\subsection*{Kolmogorov $k$-Width}

The Kolmogorov $k$-width (often referred to as the Kolmogorov $n$-width) quantifies the optimal degree to which a subset of a Banach space can be approximated by a $k$-dimensional subspace under a specified norm (see Definition \ref{def_kolmogorov_width}). Originally introduced by \cite{2ad9b75f-e5d2-3a8c-8211-9d37aa110eb2}, it has emerged as a foundational concept in approximation theory. The analytical strength of this measure lies in its generality: rather than evaluating a specific projection, it takes an infimum over all possible $k$-dimensional subspaces, thereby establishing a fundamental information-theoretic limit. As a result, it serves as an indispensable tool for deriving minimax error bounds for linear estimators, particularly when approximating functions subject to specified smoothness conditions (\cite{pinkus2012n,tsybakov2009introduction}). This paradigm later inspired the development of the Gelfand $k$-width (\cite{tikhomirov1960}), a metric that now plays a pivotal role in the theory of compressed sensing (\cite{10.1109/TIT.2006.871582}). Beyond classical statistics, the Kolmogorov $k$-width is extensively utilized in the numerical analysis of partial differential equations (PDEs), where a rapid decay rate of the width theoretically guarantees the existence of highly accurate approximate solutions (\cite{cohen2010convergence}). For modern extensions integrating this concept with neural networks and optimal transport in PDE settings, see the recent works by \cite{PAPAPICCO2022114687} and \cite{Arbes2025}.

\subsection*{Ellipsoid Method}

Independently pioneered by \cite{Shor1977} and \cite{Yudin1976}, the celebrated ellipsoid method represents a cornerstone in convex optimization. Algorithmically, it operates by querying a separation oracle at each iteration to prune a region mathematically guaranteed to exclude the global optimizer, thereby producing a sequence of enclosing ellipsoids with exponentially decaying volumes. The method achieved historic prominence when \cite{Khachiyan1979} leveraged it to establish the first polynomial-time complexity bound for linear programming. A profound theoretical consequence of this framework is the polynomial-time equivalence of the separation and optimization problems (\cite{Grotschel1981}). From the view of information-based complexity, the ellipsoid algorithm tightly characterizes the fundamental minimax lower bounds on oracle complexity for general non-smooth convex optimization (\cite{Nemirovski1983}). While it has largely been supplanted by interior-point methods in practical implementations due to their superior empirical speed, the ellipsoid method remains important in theoretical computer science for resolving the computational complexity and theoretical tractability of algorithmic problems.

\subsection{Problem Formulation and Our Contribution}\label{subsection: problem formulation and our contribution}
This paper investigates the minimax $\ell _{2}$-norm signal detection problem in the Gaussian sequence model. Departing from classical unconstrained paradigms, we impose a prior structural constraint on the mean vector, restricting it to a known set $K$, i.e., $\mu \in K$. More broadly, we formulate this problem within a robust framework, wherein the pristine observations are subject to adversarial corruption. Namely, we assume the samples can be arbitrarily modified by an adversary $\mathcal{C}$ operating under the strong $\epsilon $-contamination model, with a known upper bound $\epsilon $ on the contamination fraction.

Specifically, assume $\tilde{Y} _{i}=\mu +\sigma \cdot \xi _{i}\in \mathbb{R}^{d}, i=1,2,\dots,N$, where $\sigma >0$ is the noise scale and $\xi _{i}\sim \mathcal{N}(0,\mathbf{I}_{d})$ are standard multivariate Gaussian random variables. Denote the observations (contaminated samples) as $\mathbf{Y}:=\mathcal{C}(\tilde{\mathbf{Y}})=\left\{\mathcal{C}(\tilde{Y}_{1}),\dots,\mathcal{C}(\tilde{Y}_{N})\right\}$. The goal of the problem is to test the following hypotheses with uniform small Type I and Type II errors over $\mu \in K$ and $\mathcal{C}$ based on $\mathbf{Y}$:
\begin{equation}\label{eq: testing problem}
    \begin{aligned}
        & H _{0}: \mu =0,\\ 
        & H _{1}: \left\lVert \mu \right\rVert_{2}^{}\ge \rho ,\mu \in K.
    \end{aligned}
\end{equation}
From an information-theoretic perspective, if the signal $\mu $ is too close to the origin (i.e., the separation $\rho $ is insufficiently large), $H _{0}$ and $H _{1}$ are statistically indistinguishable, even in the pristine, uncontaminated setting. Consequently, our goal is not to solve \eqref{eq: testing problem} for all possible $\rho > 0$, but to characterize the fundamental minimal separation $\rho $ that renders the testing problem theoretically possible and computationally tractable. We investigate this fundamental limit through the minimax framework. Given a desired Type I error bound $\alpha $ (e.g., $\alpha =0.05$), we introduce the family of valid level-$\alpha $ tests, which uniformly control the false positive rate over all arbitrary adversaries $\mathcal{C}$:
\begin{equation}\label{eq: def_A_s}
A _{s}(N,d,K,\epsilon ,\alpha ,\sigma ):=\left\{\phi :\sup\limits_{\mathcal{C}}\mathbb{P}_{0}\left(\phi (\mathcal{C}(\tilde{\mathbf{Y}}))=1 \right)\le \alpha\right\},
\end{equation}
where $\mathbb{P}_{0}$ represents the distribution of the observations $\tilde{\mathbf{Y}}$ under the null hypothesis $H _{0}$. When there is no ambiguity, we suppress the explicit parameter dependence and simply write $A _{s}$.
With the robust level-$\alpha $ tests established, we define the critical minimax separation rate $\rho (N,d,K,\epsilon ,\alpha ,\sigma )$ by
\begin{equation}\label{eq: def_critical_rho}
\rho (N,d,K,\epsilon ,\alpha ,\sigma ):=\inf\limits_{\rho }\left\{\rho :\rho >0,\sup\limits_{\left\lVert \mu \right\rVert_{2}^{}\ge \rho ,\mu \in K}\sup\limits_{\mathcal{C}}\mathbb{P}_{\mu }\left(\phi (\mathcal{C}(\mathbf{Y}))=0\right)\le \alpha, \phi \in A _{s} \right\}.
\end{equation}
This rate characterizes the fundamental limit of testability: it is the smallest radius $\rho > 0$ ensuring that there exists a robust test controlling the Type II error by $\alpha $, regardless of the true signal $\mu \in K$ and the adversarial strategy $\mathcal{C}$. When no ambiguity arises, we simply denote this minimax rate by $\rho _{\text{critical}}$. This paradigm, which involves fixing the Type I error and subsequently minimizing the Type II error, is known as the Neyman-Pearson framework (\cite{05dc6f2b-c4e2-3fa9-b599-5af9d89d8396}).

In this work, we exclusively consider the \hyperref[def_epsilon_contamination_model]{strong $\epsilon $-contamination model}, which is formally introduced below.

\begin{definition}[Strong $\epsilon $-contamination model]\label{def_epsilon_contamination_model}
  Consider a set of $N$ uncorrupted observations $\tilde{\mathbf{Y}}=\left\{\tilde{Y}_{1},\dots,\tilde{Y}_{N}\right\}$ generated independently from a true underlying distribution $\mathcal{P}$ under either $H _{0}$ or $H _{1}$. For a known contamination fraction $\epsilon \in [0, \frac{1}{2})$, an omniscient adversary $\mathcal{C}$ with full knowledge of $\mathcal{P}$ is permitted to inspect the exact realization of $\tilde{\mathbf{Y}}$. The adversary may then maliciously substitute up to $\lfloor \epsilon N \rfloor$ pristine samples with arbitrary values. The resulting corrupted dataset $\mathbf{Y}=\mathcal{C}(\tilde{\mathbf{Y}})$ is provided to the statistician while the corruption index set $C$ is unknown.
\end{definition}

While the process of contamination is very common, as mentioned, $\mathcal{C}$ from the strong $\epsilon $-contamination model represents the most powerful agent for the reason that it has the full knowledge of not only the original data $\tilde{\mathbf{Y}}$, but also the data generation mechanism $\mathbb{P}_{0},\mathbb{P}_{\mu }$ and even the downstream developed algorithms. Consequently, it leads to the strongest minimax lower bounds for the testing problem \eqref{eq: testing problem}.

We note here that for the testing problem formulated above, when the constraint $K$ is a quadratically convex orthosymmetric (QCO) set (see Definition \ref{def_QCO_set} for a formal definition), the optimal minimax rate has been established by \cite{li2026robustsignaldetectionquadratically}, where the information-theoretic lower bound is matched by a theoretical algorithm and an efficient algorithm except for poly-logarithmic factors in $N,d,\frac{1}{\alpha },\frac{1}{\epsilon }$ (see Section \ref{section: testing procedures} and Appendix \ref{section: theoretical algorithm} for a comprehensive review of their results). While the efficient algorithm enjoys a polynomial-time complexity, the implementation is conditional --- the efficiency depends on the computability of the Kolmogorov $k$-width. Though the computation is feasible and efficient for some common constraint sets such as the whole space $\mathbb{R}^{d}$, axes-aligned hyperrectangles $H _{d}(a _{1},\dots,a _{d})$, ellipsoids $E _{d}(a _{1},\dots,a _{d})$ and some other QCO sets with explicit expressions, the general computability is not guaranteed efficient or even not permissible (see the discussions in Section \ref{subsection: relaxed optimization problem}). Therefore, this work aims to extend their results from the computational aspect. Specifically, we show that it is possible to relax the definition of the Kolmogorov $k$-width, and consequently reduce the computation of the Kolmogorov widths to another optimization problem where approximate and efficient solution exists. The constraints that qualify the approximation above is even beyond the QCO sets. We also show that such relaxation and approximation do not harm the original upper bounds significantly in the sense that it only introduces additional poly-logarithmic factors. In this view, the unconditional efficient algorithms for the testing problem \eqref{eq: testing problem} exist for a wide range of constraints.

\subsection{Related Work}

We review some related works in this section. The minimax testing problem in Gaussian models without adversary has been intensively studies since last century. Some pioneer works include \cite{ingster1982minimax}, where Ingster studies and establishes one of the first results in the nonparametric signal detection problem for the unconstrained case; for the constrained case, \cite{ingster1993asymptotically1} studies the testing problem with general ellipsoidal constraints; \cite{suslina1996minimax} studies the problem under the constraint of a $\ell _{q}$ ellipsoid with a $\ell _{p}$ ball removed; \cite{bj/1078435219} studies such testing problem for $\mu \in \mathbb{R}^{\infty }$ with sparse constraint, and the minimax rate was also extended to the ellipsoidal constraint. For the counterpart estimation problem, \cite{10.1214/aos/1176347758} studies the problem with the hyperrectangle constraints in the minimax sense and proved that hyperrectangles show certain least favorable property. For recent advances in the estimation problem, \cite{prasadan2025informationtheoreticlimitsrobust} establish the optimal minimax rate of estimation for star-shape constraint. \cite{neykov2026polynomialtimenearoptimalestimationcertain} establishes a computationally tractable framework for the robust estimation. While their approach similarly capitalizes on Kolmogorov width approximations, it relies on a distinctly different optimization machinery --- one that would entail an intractable exponential-time complexity if naively adapted to our hypothesis testing framework. Finally, see \cite{ingster1993asymptotically1,ingster1993asymptotically2,ingster1993asymptotically3} for the significant difference between the testing and estimation problems.

For the robust statistics, since the fundamental work by \cite{10.1214/aoms/1177703732}, various works under such framework emerged. Some recent advances include \cite{Chen2015AGD}, where the authors establish a general decision theory under the Huber contamination model and applied such theory to the high-dimensional covariance matrix using the Scheff\'e estimate, and \cite{Du2018RobustNR} where the authors construct a local binning median for the robust nonparametric regression problem and proved its minimax optimality over the H\"older function class with smoothness parameters smaller or equal to $1$. The works related to the strong $\epsilon $-contamination model (\cite{Diakonikolas2019RecentAI}) thrive after 2010s. Some important works include \cite{Narayanan2022PrivateHH,10.5555/3495724.3496571,10353143}, where the optimal minimax rates of the mean testing problem under oblivious adversaries (including the Huber contamination model) and adaptive adversaries (including the \hyperref[def_epsilon_contamination_model]{strong contamination model}) are established in the unconstrained settings. A recent work by \cite{li2026robustsignaldetectionquadratically}, as mentioned, establishes the optimal minimax rate for the robust testing problem \eqref{eq: testing problem} under the strong $\epsilon $-contamination model when the constraint $K$ is a QCO set, which includes hyperrectangles and general ellipsoids as special cases.

\subsection{Notation}

In this work, we use the following notations for consistency. We assume the constraint $K \subset \mathcal{X}=\mathbb{R}^{d}$ without special notice. Let $\tilde{\mathbf{Y}}:=\left\{\tilde{Y} _{1},\tilde{Y}_{2},\dots,\tilde{Y}_{N}\right\}$ denote the original authentic samples, where $N$ is the sample size. $\mathcal{C}$ is the adversary from the \hyperref[def_epsilon_contamination_model]{strong $\epsilon $-contamination model} and $\mathcal{C}(\tilde{\mathbf{Y}})$ is the observations contaminated by $\mathcal{C}$. Let $C$ be the index set on which the samples are contaminated. $[n]$ is used to represent the set $\left\{1,2,\dots,n\right\}$. When $I \subset [N]$ is a index set, let $\mathbf{Y}_{I}$ (or $\tilde{\mathbf{Y}}_{I}$) denote the submatrix of $\mathbf{Y}$ (or $\tilde{\mathbf{Y}}$) where the rows are selected according to $I$. By our assumptions on the contamination process, we know $\left|C\right|\le \epsilon N$ and we have $\tilde{\mathbf{Y}}_{[N]\backslash C}=\mathbf{Y}_{[N]\backslash C}$, where $\left|S\right|$ is the cardinality of $S$ when $S$ is a set. $\rho $ denotes the $\ell _{2}$ norm of the mean vector $\mu $ under the alternative. We use $\left\langle \cdot , \cdot  \right\rangle $ and $\left(\cdot \right)^{\top } \left(\cdot \right)$ alternatively to denote the inner product in a Hilbert space, probably with a subscript. The common notations $\mathcal{O},\varTheta ,\varOmega$ are used to represent the relative asymptotic orders with a possible subscript indicating the related quantity. The lowercases $c, c _{1},c _{2},\dots$ represent the constants that might vary from line to line.

\subsection{Organization}

This paper is structured as follows. Section \ref{section: linear approximation} formally introduces the concepts of the Kolmogorov $k$-width, optimal dimensions and optimal projections. This section also introduces the relaxed optimization problem related to the Kolmogorov $k$-width for the sake of the approximation later. A significant result related to the approximate solution to the constrained quadratic maximization by \cite{10.1145/3406325.3451128}, which is a crucial tool for this work, is also reviewed. Section \ref{section: approximate subgradient and ellipsoid method} introduces the modified ellipsoid method leveraged in the approximation to the Kolmogorov $k$-width mentioned above. It also contains the convergence analysis for the theoretical guarantee of the efficiency. Section \ref{section: testing procedures} is the core of the main text, where it introduces the unconditional efficient testing procedures for the problem \eqref{eq: testing problem} based on the results from Section \ref{section: linear approximation} and \ref{section: approximate subgradient and ellipsoid method} and the work by \cite{li2026robustsignaldetectionquadratically}. Finally, Section \ref{section: discussion} contains the relevant discussions and the possible directions for future works. We place the theoretical algorithm and proof of the corresponding upper bound in Appendix \ref{section: theoretical algorithm} for readers' reference. Some probabilistic preliminaries and the proofs of the lemmas, theorems, and corollaries that are not suitable for the main text are deferred to Appendix \ref{section: deferred proofs}.

\section{Linear Approximations}\label{section: linear approximation}

\subsection{\texorpdfstring{Kolmogorov $k$-Width and Optimal Dimensions}{Kolmogorov k-Width and Optimal Dimensions}}\label{subsection: Kolmogorov width}

Definition \ref{def_kolmogorov_width} below formally introduces the Kolmogorov $k$-width leveragd in this work.

\begin{definition}[Kolmogorov $k$-width]\label{def_kolmogorov_width}
    Let $\mathcal{X}$ be a Banach space equipped with the norm $\left\lVert \cdot \right\rVert_{}^{}$, and $K \subset \mathcal{X}$ is any subset. The $k$-dimensional Kolmogorov width is defined as 
    \begin{equation}\label{eq: definition of Kolmogorov width}
        D _{k}(K) = \inf_{P \in \mathcal{P}_{k}} \sup_{\theta \in K} \left\lVert \theta -P \theta \right\rVert_{}^{},
    \end{equation}
    where $\mathcal{P}_{k}$ is the set of all projection operators that project a vector onto a subspace in $\mathcal{X}$ with intrinsic dimension $k$.
\end{definition}

\begin{remark*}[]
  For a fixed $\theta \in \mathcal{X}$, the ``projection operator'' in the definition above should be understood as finding the point $P \theta \in K$ such that the the distance between $\theta $ and $P \theta $ is (nearly) minimal. When the norm in \eqref{eq: definition of Kolmogorov width} is the $\ell _{2}$ norm, it can be characterized by some projection matrix.
\end{remark*}

Throughout this work, we operate in the Euclidean space $\mathcal{X}=\mathbb{R}^{d}$ and define the Kolmogorov width with respect to the $\ell _{2}$-norm, in line with standard practice. We also restrict the candidate set $\mathcal{P}_{k}$ to orthogonal projections. We further note that in the $\ell _{2}$ norm separation, the Kolmogorov width is defined originally as in Definition \ref{def_kolmogorov_width}. However, in the case of $\ell _{p}$ norm separation, such definition should be discretized, i.e., only consider the projections that are aligned with the axes. This topic is beyond the scope of this work, and we kindly refer the reader to Section 4 of \cite{li2026robustsignaldetectionquadratically} for more details.

By definition, the Kolmogorov width $D _{k}(K)$ is monotonically non-increasing with respect to the subspace dimension $k$. A crucial determinant for characterizing the geometric and statistical properties of the subset $K$ within the ambient space $\mathcal{X}$ is the exact rate of this decay. Intuitively, a rapid decay rate implies that $K$ can be efficiently approximated by low-dimensional subspaces, reflecting an intrinsically lower structural complexity. For instance, the class of analytic functions on a compact interval exhibits an exponential decay rate, bounded by $D _{k}(K)\asymp e ^{-ck}$ for some absolute constant $c>0$. Conversely, a slow decay rate indicates that the set resists efficient linear approximation. A canonical example of this is the Sobolev space $W _{p}^{r}$, whose Kolmogorov width decays polynomially as $D _{k}(K)\asymp k ^{-\alpha }$, where the exponent $\alpha $ is governed by the smoothness parameter of the space. For a comprehensive exposition on the exact calculation of Kolmogorov widths for various reproducing kernel Hilbert spaces and Sobolev spaces, we refer the reader to \cite{pinkus2012n}.

In the context of statistical inference and hypothesis testing, however, the intrinsic geometric properties of the subset $K$ and the whole space $\mathcal{X}$ are of secondary interest compared to their interaction with the true data-generating mechanism. This distinction, coupled with the concept of approximation decay rates, motivates us to define new complexity measures. These quantities are specifically designed to bridge the geometric structure of the parameter space—characterized by its decay rate—with the statistical nature of the underlying distributions.

\begin{definition}[Optimal dimension]\label{def_optimal_dimension}
  Assume $\tilde{\mathbf{Y}}$ are drawn from the distribution $\mathcal{P}\in \mathscr{P}$. Let the function $f(\cdot ,\cdot ):\mathbb{N}\times \mathscr{P}\mapsto \mathbb{R}^{+}$, then a optimal dimension regarding $K$ and $f$ is defined as 
  \begin{equation*}
    k _{f} ^{\star }:=\max\limits \left\{j \left\lvert\right. j \ge 0,D _{j-1}(K)>f(j,\mathcal{P})\right\},
  \end{equation*}
  where we additionally define $D _{k}(K)=\infty $ for $k<0$ and $D _{k}(K)=0$ when $k \ge d=\dim (\mathcal{X})$. The corrsponding optimal projection with dimension $k$ is denoted as $P _{k}^{\star }$.
\end{definition}

\begin{remark*}[]
  Such definition can be extended to $\ell _{2}$ with minimal effort.
\end{remark*}

The index $j-1$ in the definition is set for the convenience of the computation and is not essentially important. Since $D _{k}(K)=0$ for $k \ge d$ and $D _{0}(K)=\sup\limits _{\theta \in K}\left\lVert \theta \right\rVert_{2}^{}>0$ when $K \neq \left\{\mathbf{0}  \right\}$, such $k _{f}^{\star }$ always exists given $K \neq \left\{\mathbf{0}\right\}$. It is notable that $D _{k}(K)$ is irrelevant with $\tilde{\mathbf{Y}}$, $\mathbf{Y}$, and $\mathcal{P}$, while $f(j,\mathcal{P})$ is irrelevant with $\tilde{\mathbf{Y}}$, $\mathbf{Y}$, $\mathcal{X}$ and $K$. Therefore, $k _{f}^{\star }$ is a quantity characterizing the interaction between $K$, $\mathcal{X}$, and $\mathcal{P}$, and does not depend on the realization $\tilde{\mathbf{Y}}$ and $\mathbf{Y}$ once $K$, $\mathcal{X}$ and $\mathcal{P}$ are assumed.

The optimal dimension is closely related to the relevant minimax problem. In \cite{li2026robustsignaldetectionquadratically}, the authors set $f (j,\mathcal{P})$ as $\frac{j ^{1/4}}{\sqrt{N}}\sigma $ and $\frac{j ^{1/4}\sqrt{\epsilon }}{N ^{1/4}}\sigma $ in Definition \ref{def_optimal_dimension}, respectively, and the resulting optimal dimensions (denoted as $k _{1}^{\star }$ and $k _{2}^{\star }$) are crucial for the derived optimal minimax rate for QCO constraints. See Section \ref{subsection: a simple case} and \ref{subsection: robust testing} for more details; In \cite{prasadan2025informationtheoreticlimitsrobust}, the authors obtain the optimal minimax rate for the $\ell _{2}$ robust estimation problem with star-shape constraints. Though they do not explicitly specify the optimal dimension, there results can be summarized via such concept. Let $f (j,\mathcal{P})=\sqrt{\frac{j}{N}}\sigma $ in Definition \ref{def_optimal_dimension} and denote the resulting optimal dimension by $k _{e}^{\star }$. \cite{prasadan2025informationtheoreticlimitsrobust} prove that 
\begin{equation}\label{eq: lower bound of estimation}
    \inf\limits_{\hat{\nu }} \sup\limits_{\mu \in K} \mathbb{E}_\mu \left\lVert \hat{\nu }(\mathbf{Y})-\mu \right\rVert_{2}^{2}\asymp \sigma ^{2}\max\limits \left\{\frac{k _{e}^{\star }}{N},\epsilon ^{2}\right\},
\end{equation}
where $\inf $ is taken with respect to all measurable functions $\hat \nu$ of the observations $\mathbf{Y}$.

\subsection{Relaxed Optimization Problem}\label{subsection: relaxed optimization problem}

When $K$ is an ellipsoid defined by a positive semidefinite matrix and the distance is measured under the Euclidean norm, calculating the optimal $k$-dimensional approximation is equivalent to principal component analysis (PCA). However, computing the Kolmogorov $k$-width for general sets and arbitrary norms is a formidable task, typically incurring a super-polynomial time complexity.

As an example in $\mathbb{R}^{d}$, \cite{brieden2002geometric} demonstrate that computing $D_0(K)$ for convex polytopes (even symmetric ones) cannot be approximated in polynomial time within a factor of $1.090$ unless $\text{P}=\text{NP}$. Even when restricting the metric to the $\ell _{2}$-norm, the geometric representation of $K$ can be highly complex; in certain scenarios, $K$ might only be accessible via a membership oracle rather than an explicit analytical expression, further exacerbating the computational intractability of the Kolmogorov widths.

To circumvent the computational intractability of exactly evaluating the Kolmogorov widths and identifying their associated optimal projections, we propose a computationally efficient alternative. In particular, we formulate the following semidefinite programming (SDP) relaxation, which allows us to tightly approximate the Kolmogorov widths in polynomial time.

\begin{definition}[Approximate Kolmogorov $k$-Width]\label{def_approximate_kolmogorov_width}
  Let $\mathcal{X}=\mathbb{R}^{d}$ equipped with the Euclidean norm. For any $K \subset \mathcal{X}$ and $0 \le k \le d$, we define the approximate Kolmogorov $k$-width $\tilde{D}_{k}^{2}(K)$ as the optimal value of the following SDP.
  \begin{equation}\label{eq: SDP problem}
  \begin{aligned}
    & \min\limits _{X \in \mathbb{R}^{d \times d}}\max\limits _{\theta \in K}\theta ^{\top } X \theta,\\ 
    & \text{subject to } \text{tr}(X)=d-k, \mathbf{0}\preceq X \preceq \mathbf{I}_{d}, X ^{\top } =X.
  \end{aligned}
\end{equation}
We further denote the solution to \eqref{eq: SDP problem} as $X ^{\star } _{k}$.
\end{definition}

Indeed, \eqref{eq: SDP problem} is a relaxed optimization problem compared with the original definition by the following fact.

\begin{lemma}\label{lemma: approximate Kolmogorov width}
  For any $0 \le k \le d$ and $\theta \in K$, we have $\theta \T X ^{\star }_{k}\theta \le \tilde{D}_{k}^{2}(K)\le D^2_{k}(K)$.
\end{lemma}

\begin{proof}[Proof of Lemma \ref{lemma: approximate Kolmogorov width}]
  For any fixed $k$, let $P_{d-k}:=\mathbf{I}_{d}-P _{k}^{\star }$, where $P _{k}^{\star }$ satisfies $\sup_{\theta \in K} \left\lVert (\mathbf{I}_{d}-P _{k})\theta \right\rVert_{2}^{2}= D^2_k(K)$. This is equivalent to $\theta ^{\top } P_{d-k}\theta \le D^2_k(K)$. Since $P_{d-k}$ is a feasible point of \eqref{eq: SDP problem}, we know $\tilde{D}_{k}^{2}(K)\le D _{k}^{2}(K)$. The other inequality follows directly from the definition of \eqref{eq: SDP problem}.
\end{proof}

We should be aware that $\tilde{D}_{k}(K)$ is possibly significantly smaller than $D _{k}(K)$. Hence, \eqref{eq: SDP problem} is a non-trivial relaxation.

\begin{example}[]
  Consider the case when $K=B _{d}(0,1)$ i.e., the unit ball in $\mathbb{R}^{d}$. It is not hard to see that 
  \begin{equation*}
    D _{k}(B _{d}(0,1))=\left\{
    \begin{aligned}
      1 &\ (0 \le k \le d-1),\\ 
      0 &\ (k=d).
    \end{aligned}
    \right.
  \end{equation*}
  However, we have $\tilde{D}_{k}(K)=\min\limits _{X}\sqrt{\lambda _{\text{max}}(X)}$ subject to $\text{tr}(X)=d-k, \mathbf{0}\preceq X \preceq \mathbf{I}_{d}, X=X ^{\top }$. Minimal calculation shows that $\tilde{D}_{k}(K)=\sqrt{1-\frac{k}{d}}$. The ratio between $D _{k}(K)$ and $\tilde{D}_{k}(K)$ is $\sqrt{d}\gtrsim 1$ when $k=d-1$.
\end{example}

The SDP \eqref{eq: SDP problem} formulated above exhibits a bilevel optimization structure with respect to $X$ and $\theta $. Focusing on the inner problem, for any fixed feasible matrix $X \in \mathbb{R}^{d \times d}$, we aim to solve the following classical constrained quadratic maximization:
\begin{equation}\label{eq: inner optimization problem}
\begin{aligned}
&\max\limits _{\theta }\theta  ^{\top } X \theta ,\
&\text{subject to }\theta \in K.
\end{aligned}
\end{equation}
While \eqref{eq: inner optimization problem} is conventionally formulated over an origin-symmetric convex set $K$, identifying its exact global optimum is analytically intractable. Even in the seemingly benign case where $X$ is positive semidefinite and $K$ is a polytope in $\mathbb{R}^{d}$, maximizing a convex function over a polytope remains notoriously NP-hard due to the potentially exponential growth of the number of vertices with respect to $d$. To circumvent this, \cite{10.1145/3406325.3451128} recently proposed a polynomial-time approximation framework based on a modified ellipsoid method (which is distinct from the algorithm discussed later in Section \ref{section: approximate subgradient and ellipsoid method}). This framework ensures that a $\kappa $-approximate optimal value for \eqref{eq: inner optimization problem} can be efficiently obtained. Crucially, \cite{10.1145/3406325.3451128} establish that satisfying the type-$2$ condition on the set $K$ is a strict prerequisite for the tractability of this approximation.
\begin{definition}[Minkowski gauge]\label{def_minkowski_gauge}
  Let $\mathcal{X}$ be a vector space $\mathcal{X}$ and $K \subset \mathcal{X}$ is any fixed subset. The Minkowski gauge of $K$ is defined as a function $\rho _{K}:\mathcal{X}\rightarrow [0,\infty ]$, where
  \begin{equation}
    \rho _{K}(\theta )=\inf\limits _{r}\left\{r \left\lvert\right. r \in \mathbb{R}^{+},\theta  \in rK,\theta  \in \mathcal{X}\right\}.
  \end{equation}
  If such $r$ does not exist for a given $\theta $, we exclusively define $\rho _{K}(\theta )=\infty $.
\end{definition}
\begin{definition}[Type-2 condition]\label{def_type_2_condition}
  Let $\left\{\epsilon _{i}\right\}_{i=1}^{m}$ be i.i.d. Rademacher random variables for some $m \in \mathbb{N}^{+}$. For $K \subset \mathcal{X}$ above, $K$ is type-$2$ with constant $T _{2}(K)$ if for any $m \in \mathbb{N}^{+}$ and any $\theta _{i}\in \mathbb{R}^{d}$, we have 
  \begin{equation}\label{eq: required condition in type 2 definition}
    \mathbb{E}_{\epsilon }\rho _{K}^{2}\left(\sum\limits_{i=1}^{m}\epsilon _{i}\theta _{i}\right)\le T _{2}^{2}(K)\sum\limits_{i=1}^{m}\rho _{K}^{2}(\theta _{i}).
  \end{equation}
  $T _{2}^{2}(K)$ is referred as the type-$2$ constant of $K$.
\end{definition}

\begin{example}[$\ell _{p}$ unit ball]\label{example: t_2(K) of unit ball in l_p norm}
  \cite{Milman1986AsymptoticTO} establish the following results about the $\ell _{p}$ unit ball (i.e., $\rho _{K}=\left\lVert \cdot \right\rVert_{p}^{}$) in $\mathbb{R}^{d}$.
  \begin{equation*}
    T _{2}(K)\asymp \left\{
    \begin{tabular}{ll}
      $d ^{\frac{1}{p}-\frac{1}{2}}$, & $(1 \le p \le 2)$,\\ 
      $\sqrt{\min\limits \left\{p,\ln d\right\}}$, & $(2 \le p \le \infty )$
    \end{tabular}
    \right.
  \end{equation*}
\end{example}

Crucially, \cite{10.1145/3406325.3451128} establish that the originally intractable problem \eqref{eq: inner optimization problem} can be rigorously reduced to the following equivalent semidefinite program:
\begin{equation}\label{eq: equivalent SDP problem}
  \begin{aligned}
    &\max\limits _{\mathbf{\Theta }}\left\langle X, \mathbf{\Theta } \right\rangle ,\\
    &\text{subject to }\mathbb{E}\left[\rho _{K}^{2}(\mathbf{\Theta }^{\frac{1}{2}}g)\right]\le 1, \mathbf{\Theta }\succeq \mathbf{0},
  \end{aligned}
\end{equation}
where $g \in \mathbb{R}^{d}$ denotes a standard Gaussian random vector.
\begin{proposition}[Observation 4.1 of \cite{10.1145/3406325.3451128}]\label{proposition: equivalence from the original SDP problem}
  The optimal value of the original SDP \eqref{eq: inner optimization problem} is equivalent to that of \eqref{eq: equivalent SDP problem}.
\end{proposition}
\begin{proof}
  Let $f _{ori}(A)$ denote the optimal value of \eqref{eq: inner optimization problem}, $f _{\text{equ}}(A)$ denote the optimal value of \eqref{eq: equivalent SDP problem}, and $\theta _{1}\in K$ be any vector such that $\theta ^{\top } A \theta =f _{\text{raw}}(A)$. Consider $\mathbf{\Theta }=\theta _{1}\theta _{1}^{\top } $. Since $\mathbb{E}\rho _{K}^{2}(\mathbf{\Theta }^{\frac{1}{2}}g)=\mathbb{E}\rho _{K}^{2}\left(\frac{\theta _{1}X _{1}^{\top } g}{\left\lVert \theta _{1}\right\rVert_{2}^{}}\right)=\mathbb{E}\left(\theta _{1}^{\top } g\right)^{2}\cdot \rho _{K}^{2}\left(\frac{\theta _{1}}{\left\lVert \theta _{1}\right\rVert_{2}^{}}\right)=\left\lVert \theta _{1}\right\rVert_{2}^{2}\cdot \rho _{K}^{2}\left(\frac{\theta _{1}}{\left\lVert \theta _{1}\right\rVert_{2}^{2}}\right)=\rho _{K}^{2}(\theta)\le 1$, we know that $f _{\text{ori}}(A)\le f _{\text{equ}}(A)$.

  On the other hand, for any $\mathbf{X}$ that is feasible for \eqref{eq: equivalent SDP problem}, we have
  \begin{equation*}
    \left\langle A, \mathbf{\Theta } \right\rangle =\mathbb{E}_{g}\left\langle \mathbf{\Theta }^{\frac{1}{2}}g, A \mathbf{\Theta }^{\frac{1}{2}}g \right\rangle =\rho _{K}^{2}(\mathbf{\Theta }^{\frac{1}{2}}g)\cdot \mathbb{E}_{g}\left(\frac{g ^{\top } \mathbf{\Theta }^{\frac{1}{2}}}{\rho _{K}(\mathbf{\Theta }^{\frac{1}{2}}g)}A \frac{\mathbf{\Theta }^{\frac{1}{2}}g}{\rho _{K}(\mathbf{\Theta }^{\frac{1}{2}}g)}\right)\overset{\text{(i)}}{\le } \rho _{K}^{2}(\mathbf{\Theta }^{\frac{1}{2}}g)f _{\text{ori}}(A)\le f _{\text{ori}}(A),
  \end{equation*}
  where (i) is from the fact that $\frac{\mathbf{\Theta }^{1/2}g}{\rho _{K}(\mathbf{\Theta }^{1/2}g)}\in K$ by the definition of the \hyperref[def_minkowski_gauge]{Minkowski gauge}. Taking the supreme over $\mathbf{\Theta }$ on the LHS completes the proof.
\end{proof}

As a final preliminary, we introduce the notions of sign-invariant norms, exact 2-convexity, and balanced set .

\begin{definition}[Sign-invariant norm]\label{def_sign-invariant_norm}
  A norm $\left\lVert \cdot \right\rVert_{}^{}$ is \textit{sign-invariant} if for any $\theta \in \mathbb{R}^{d}$ and set of signs $\gamma =(\gamma _{1},\dots,\gamma _{d}), \gamma _{i}\in \left\{-1,1\right\}, i=1,\dots,d$, we have $\left\lVert X\right\rVert_{}^{}=\left\lVert \gamma \odot \theta \right\rVert_{}^{}$, where $\odot $ denotes the entrywise multiplication.
\end{definition}

\begin{definition}[Exact 2-convex norm]\label{def_exact_2_convex_norm}
  The Minkowski gauge $\rho _{K}(\cdot )$ of a set $K$ is 2-convex if for any $m \in \mathbb{N}^{+}$ and $\theta _{i},\dots,\theta _{m}\in \mathbb{R}^{d}$, we have 
  \begin{equation*}
    \rho _{K}\left(\left(\sum\limits_{i=1}^{m}\theta _{i}^{2}\right)^{\frac{1}{2}}\right)\le c(K) \left(\sum\limits_{i=1}^{m}\rho _{K}^{2}(\theta _{i})\right)^{\frac{1}{2}},
  \end{equation*}
  where $\theta ^{2}=\left(\theta _{1}^{2},\dots,\theta _{d}^{2}\right)^{\top } $ is the entrywise square of $\theta$ and $c(K)$ is a quantity that only depends on $K$. When $c(K)=1$, $\rho _{K}(\cdot )$ is said to be \textit{exact 2-convex}.
\end{definition}

\begin{definition}[Balanced set]
  For a subset $K \subset \mathcal{X}=\mathbb{R}^{d}$, $K$ is balanced if there exist $0<r \le R \le \infty $ such that $B _{d}(0,r)\subset K \subset B _{d}(0,R)$.
\end{definition}

The following theorem from \cite{10.1145/3406325.3451128} formally establishes the existence of an approximate efficient solution to \eqref{eq: inner optimization problem}.

\begin{theorem}[Proposition 4.2 and Theorem 7.6 of \cite{10.1145/3406325.3451128}]\label{theorem: main theorem of the cited paper}
  For the equivalent semidefinite program \eqref{eq: equivalent SDP problem}, if $K$ is type-$2$ and balanced, and the Minkowski gauge $\rho _{K}$ is a sign-invariant and exact 2-convex norm, then there is an algorithm $\mathcal{A}(X,R,r)$, on any input $X \in \mathbb{R}^{d \times d}$, $\mathcal{A}$ runs in polynomial time of $d,\left|\ln R\right|,\left|\ln r\right|,\text{bit}(X)$ and outputs a vector $\theta _{1}\in K$ such that $\theta _{1}^{\top } X\theta _{1}\gtrsim \frac{1}{\kappa }\max\limits _{\theta \in K}\theta ^{\top } X \theta$ with $\kappa \lesssim \ln d$.
\end{theorem}

\begin{remark*}[]
  Strictly speaking, the main conclusion of Theorem \ref{theorem: main theorem of the cited paper} holds only with high probability. However, since the algorithm terminates within a polynomial number of iterations --- as established by the convergence guarantees of the ellipsoid method in Theorem \ref{theorem: convergence guarantee of approximate ellipsoid method} later --- the accumulated probability of failure remains well-controlled and has a negligible effect on the overall procedure. For the sake of brevity, we omit the detailed argument here. Interested readers are referred to \cite{neykov2026polynomialtimenearoptimalestimationcertain} for a relevant discussion in an analogous setting.
\end{remark*}

For a comprehensive treatment of this framework—encompassing Theorem \ref{theorem: main theorem of the cited paper} and the underlying proof techniques—we refer the reader to \cite{10.1145/3406325.3451128}. Moving forward, we operate under the assumption that $K$ is type-$2$ and balanced, and its associated Minkowski gauge is sign-invariant and exactly 2-convex. As we establish in Lemma \ref{lemma: QCO sets are type-2 and exactly 2-convex}, the QCO sets inherently satisfy these structural prerequisites and thus emerge as a natural special case.

Back to our problem \eqref{eq: SDP problem}. For the inner problem, we leverage Theorem \ref{theorem: main theorem of the cited paper} and know that for any feasible $X$ in \eqref{eq: SDP problem}, there exists an oracle $\mathcal{O}(X)\in K$ guaranteed by Theorem \ref{theorem: main theorem of the cited paper} that approximates the optimal value of \eqref{eq: SDP problem} by a constant $\kappa $:
\begin{equation}\label{eq: property of the oracle}
  \mathcal{O}(X)^{\top } X \mathcal{O}(X)\ge \frac{1}{\kappa }\max\limits _{\theta \in K}\theta ^{\top } X \theta .
\end{equation}

Importantly, the approximation factor $\kappa $ remains moderately small, being bounded strictly by a polylogarithmic function of the dimension $d$, i.e., $\kappa =\text{polylog}(d)$.

Concluding this preparatory section, we introduce the formal definition of QCO sets and establish that they naturally fulfill the prerequisites of Theorem \ref{theorem: main theorem of the cited paper}.

\begin{definition}[Quadratically convex orthosymmetric (QCO) set]\label{def_QCO_set}
    Given a set $K \subset \mathbb{R}^{d}$, $K$ is a QCO set if it satisfies the following conditions:\\ 
\hspace*{0.5em}(1), $K$ is convex;\\ 
\hspace*{0.5em}(2), $K$ is quadratically convex, i.e., $K ^{2}$ is also convex, where $K ^{2}$ is defined as
\begin{equation*}
    K ^{2}:=\left\{\left(\theta _{1}^{2},\dots,\theta _{d}^{2}\right)^{\top } \left\lvert\right. \left(\theta _{1},\dots,\theta _{d}\right)^{\top } \in K\right\};
\end{equation*} 
\hspace*{0.5em}(3), $K$ is orthosymmetric, i.e., if $\theta =\left(\theta _{1},\dots,\theta _{d}\right)^{\top } \in K$, then for any $\eta =(\eta _{1},\dots,\eta _{d})^{\top } $, $\eta _{i}\in \left\{-1,1\right\}, i=1,\dots,d$, we have $\theta _{\eta }:=\eta \odot \theta \in K$.
\end{definition}

\begin{lemma}[\cite{neykov2026polynomialtimenearoptimalestimationcertain}]\label{lemma: QCO sets are type-2 and exactly 2-convex}
  For any QCO set $K$, $K$ is type-$2$ and exactly $2$-convex. Furthermore, the type-$2$ constant $T _{2}(K)=c \ln d$ for some universal constant $c$.
\end{lemma}
The proof of Lemma \ref{lemma: QCO sets are type-2 and exactly 2-convex} is essentially from \cite{neykov2026polynomialtimenearoptimalestimationcertain}. We place the proof of Lemma \ref{lemma: QCO sets are type-2 and exactly 2-convex} in Appendix \ref{subsection: proof of the property that QCO sets are type-2 and exactly 2-convex} for the reader's convenience.

\section{Approximate Subgradient and Ellipsoid Method}\label{section: approximate subgradient and ellipsoid method}

\subsection{Approximate Subgradient and Related Property}

In the context of the SDP \eqref{eq: SDP problem}, we define the optimal value function $h _{\text{S}}(X):=\max\limits _{\theta \in K}\theta ^{\top } X \theta $ for any feasible $X$, and denote its corresponding maximizer by $\theta ^{\star }_{X}$ (provided it exists). It is straightforward to verify that the outer product $\theta ^{\star }_{X}\theta ^{\star \top }_{X}$ serves as a valid subgradient of $h _{S}$ at $X$. This directly follows from the fact that for any feasible $X _{1}, X _{2}$, the following inequality holds: $h _{S}(X _{2})=\theta _{X _{2}}^{\star \top } X _{2}\theta _{X _{2}}^{\star }\ge \theta _{X _{1}}^{\star \top }X _{2}\theta _{X _{1}}^{\star }=h _{S}(X _{1})+\left\langle \theta _{X _{1}}^{\star }\theta _{X _{1}}^{\star \top }, X _{2}-X _{1}\right\rangle$. Consequently, since the oracle $\mathcal{O}(X)$ only provides an approximation of the optimal value, we anticipate it to return an approximate subgradient, thereby satisfying a relaxed subgradient condition.

\begin{lemma}[]\label{lemma: approximate subgradient}
  For any feasible $X$ and the corresponding $\mathcal{O}(X)$, $\mathcal{O}(X)\mathcal{O}(X)^{\top } $ is an $\omega $-subgradient of $h _{S}$ at $X$ in the sense that for any feasible $X ^{\prime }$, we have 
  \begin{equation}\label{eq: property of approximate subgradient}
    h _{S}(X ^{\prime })\ge \frac{1}{\kappa }h _{S}(X)+\mathcal{O}(X)^{\top } (X ^{\prime }-X)\mathcal{O}(X).
  \end{equation}
\end{lemma}

\begin{proof}
  Lemma \ref{lemma: approximate subgradient} can be directly verified via \eqref{eq: property of the oracle}. We have 
  \begin{equation*}
    \begin{aligned}
      h _{S}(X)+\mathcal{O}(X)^{\top } (X ^{\prime }-X)\mathcal{O}(X)&=h _{S}(X)-\mathcal{O}(X)^{\top } X \mathcal{O}(X)+\mathcal{O}(X)^{\top } X ^{\prime }\mathcal{O}(X)\\ 
      &\le \left(1-\frac{1}{\kappa }\right)h _{S}(X)+h _{S}(X ^{\prime }).
    \end{aligned}
  \end{equation*}
  Cancelling $h _{S}(X)$ on the both sides yields the results.
\end{proof}
Lemma \ref{lemma: approximate subgradient} tells us that while $\mathcal{O}(X)\mathcal{O}(X)^{\top } $ is no longer a rigorous subgradient, it preserves the property somehow by adding a $\frac{1}{\kappa }$ factor, which is the key to the following ellipsoid method.

\subsection{Ellipsoid Method}

\subsubsection{Problem without Equality Constraints}

The ellipsoid method (\cite{Shor1977,Yudin1976}) stands as a fundamental iterative tool for constrained convex optimization. Diverging from standard gradient-descent techniques that take local steps towards the optimal solution, the ellipsoid method systematically constructs a sequence of bounding ellipsoids strictly guaranteed to enclose the true minimizer. The algorithm derives its convergence from the continuous, geometric reduction in the volume of the ellipsoids. In this theoretical light, the approach can be understood as a generalization of the one-dimensional bisection method to the Euclidean space $\mathbb{R}^{d}$. Crucially, the exponential decay rate of these ellipsoidal volumes is precisely what guarantees polynomial-time computational complexity across a broad spectrum of optimization paradigms.

Consider the unconstrained convex optimization problem:
\begin{equation}\label{eq: simple optimization problem}
\begin{aligned}
& \text{minimize } f(x), x \in \mathbb{R}^{d},
\end{aligned}
\end{equation}
where $f: \mathbb{R}^{d} \to \mathbb{R}$ ($d \ge 2$) is convex. We do not require $f$ to be differentiable, assuming only the availability of a subgradient oracle. Let the algorithm be initialized with an ellipsoid $E ^{(0)}=\left\{x \left\lvert\right. (x-x ^{(0)})^{\top } \left(P ^{(0)}\right)^{-1}(x-x ^{(0)})\le 1\right\}$ centered at $x ^{(0)} \in \mathbb{R}^{d}$, where $P ^{(0)}\in \mathbb{R}^{d \times d}$ is a positive definite matrix. At the $k$-th iteration, let $x ^{(k)}$ and $P ^{(k)}$ define the current search region, and compute a subgradient $g ^{(k)} \in \partial f(x ^{(k)})$. If $\mathbf{0} \in \partial f(x ^{(k)})$, the exact global minimizer is successfully identified at $x ^{(k)}$. Otherwise, for any non-zero $g ^{(k)}$, the subgradient inequality implies that any point $x$ satisfying $\left\langle x-x ^{(k)}, g ^{(k)} \right\rangle > 0$ strictly degrades the objective value relative to $f(x ^{(k)})$, fundamentally excluding such points from optimality. Consequently, our objective is to construct an updated ellipsoid $E ^{(k+1)}$, parameterized by $x ^{(k+1)}$ and $P ^{(k+1)}$, that tightly circumscribes the remaining feasible half-ellipsoid:
\begin{equation*}
\left\{x \left\lvert\right. x \in E ^{(k)}, \left\langle x-x ^{(k)}, g ^{(k)} \right\rangle \le 0\right\}\subset E ^{(k+1)}:=\left\{x \left\lvert\right. (x-x ^{(k+1)})^{\top } \left[P ^{(k+1)}\right]^{-1} (x-x ^{(k+1)})\le 1\right\}.
\end{equation*}
We select the ellipsoid with smallest volume while satisfying the condition above. To determine such ellipsoid, we can first assume $E ^{(k)}=\mathbf{I}_{d}$, and transform back with the affine transformation $\left(E ^{(k)}\right)^{\frac{1}{2}}$. It is proved (see the original work by \cite{yudin1976informational}) that the ellipsoid in the $(k+1)$-th iteration is determined by:
\begin{equation}\label{eq: update of the ellipsoid method without constraint}
  \begin{aligned}
    & x ^{(k+1)}=x-\frac{P ^{(k)}g ^{(k)}}{(d+1)\sqrt{g ^{(k)\top }P ^{(k)} g ^{(k)}}},\\ 
    & P ^{(k+1)}=\frac{d ^{2}}{d ^{2}-1}\left(P ^{(k)}-\frac{2}{d+1}\frac{P ^{(k)} g ^{(k)}g ^{(k)\top } P ^{(k)}}{g ^{(k)\top }P ^{(k)}g ^{(k)}}\right).
  \end{aligned}
\end{equation}
The updates above can be generalized to convex optimization programs with inequality constraints. Consider the following constrained problem:
\begin{equation}\label{eq: optimization problem with inequality constraints}
\begin{aligned}
&\text{minimize } f _{0}(x),\
&\text{subject to } f _{i}(x)\le 0, i=1,\dots,m,
\end{aligned}
\end{equation}
where the functions $f _{i}$ for $i=0,1,\dots,m$ are convex and equipped with subgradient oracles. Guided by the same geometric intuition, the update procedure at any iteration hinges strictly on the feasibility of the current iterate $x ^{(k)}$. In particular: (1) if $x ^{(k)}$ violates any constraint, we identify a violated index (e.g., the smallest $i$ satisfying $f _{i}(x ^{(k)})>0$) and execute the update using a \textit{feasibility cut} derived from the subgradient of $f _{i}$; (2) if $x ^{(k)}$ resides within the feasible region, we proceed with an \textit{objective cut} based on the subgradient of $f _{0}$, mirroring the unconstrained scenario. Thus, from the perspective of the ellipsoid method, constrained optimization is mechanistically equivalent to its unconstrained counterpart; in both paradigms, the algorithm consistently curtails the search space by constructing a ellipsoid with smaller volume.

\subsubsection{Convergence Analysis}
As noted earlier, the theoretical convergence of the ellipsoid algorithm relies on the monotonic decrease in the volume of the constructed ellipsoids, namely, $\text{Vol}(E ^{(k+1)})<\text{Vol}(E ^{(k)})$. In fact, we have the following famous results

\begin{theorem}[]\label{theorem: convergence guarantee of vanilla ellipsoid method}
  For the convex optimization problem \eqref{eq: simple optimization problem} and \eqref{eq: optimization problem with inequality constraints}, we have
  \begin{equation*}
    \frac{\text{Vol}(E ^{(k+1)})}{\text{Vol}(E ^{(k)})}=\frac{d ^{d}}{(d+1)^{\frac{d+1}{2}}\cdot (d-1)^{\frac{d-1}{2}}}<e ^{-\frac{1}{2d}}.
  \end{equation*}
  Furthermore, for the feasible $\epsilon $-suboptimal set
  \begin{equation*}
    \left\{x \left\lvert\right. x \in \mathbb{R}^{d},f _{0}(x)-f _{0}(x ^{\star })\le \epsilon ,f _{i}(x)\le 0,i=1,\dots,m\right\}
  \end{equation*}
  where $x ^{\star }$ is the true minimizer, if it has positive volume, then we have
  \begin{equation*}
    \left|f(x ^{\dagger })-f(x ^{\star })\right|\le c \cdot \exp\left\{-\frac{M}{2d ^{2}}\right\},
  \end{equation*}
  where $x ^{\dagger }_{M}:=\mathop{\arg \min}\limits _{x ^{(i)},i=1,2\dots,M}f(x ^{(i)})$ is the best candidate up to the $M$-th iteration, and $c$ is a constant that only depends on $f _{0}$.
\end{theorem}

From Theorem \ref{theorem: convergence guarantee of vanilla ellipsoid method}, we know that the ellipsoid method converges exponentially under very mild conditions. Equivalently, for any given precision $\epsilon $, the ellipsoid method finds at least one feasible point that is $\epsilon $-suboptimal in at most $c ^{\prime }d ^{2}\ln \left(\frac{1}{\epsilon }\right)$ iterations for some constant $c ^{\prime }$. The proof of Theorem \ref{theorem: convergence guarantee of vanilla ellipsoid method} derives from the updates \ref{eq: update of the ellipsoid method without constraint} and the computation of their volumes, which can be found in many references (\cite{Shor1970ConvergenceRO, zbMATH03516928, KHACHIYAN198053}). Note that when computing the volume ratio between $E ^{(k+1)}$ and $E ^{(k)}$, we can again assume $x ^{(k)}=\mathbf{0}, P ^{(k)}=\mathbf{I}_{d}$ since the volume ratio is invariant under the affine transformations.

The pseudo code of the ellipsoid method for the problems \eqref{eq: simple optimization problem} and \eqref{eq: optimization problem with inequality constraints} is provided as follows.
\begin{algorithm}[htbp]
    \caption{ellipsoid method with (inequality) constraints}\label{algorithm: ellipsoid method with inequality constraints}

    \KwInput{$d, f _{0},...,f _{m},x ^{(0)},P ^{(0)},\epsilon ,c ^{\prime }$}

    $k=0, x _{\text{best}}=\mathbf{0}, v _{\text{best}}=\infty $

    $M=c ^{\prime }d _{A}^{2}\ln \left(\frac{1}{\epsilon }\right)$

    \While{$k<M$}{
      \If{$x ^{(k)}$ is feasible}{

        \If{$f(x ^{(k)})<v _{\text{best}}$}{
          $x _{\text{best}}=x ^{(k)}$

          $v _{\text{best}}=f(x ^{(k)})$
        }

        Select any $g ^{(k)}\in \partial f _{0}(x ^{(k)})$
      }
      \Else{
        Find the smallest $i$ such that $f _{i}(x ^{(k)})> 0$

        Select any $g ^{(k)}\in \partial f _{i}(x ^{(k)})$
      }
      
      Update $x ^{(k+1)}=x ^{(k)}-\frac{P ^{(k)}g ^{(k)}}{(d+1)\sqrt{g ^{(k)\top }P ^{(k)} g ^{(k)}}}$

      Update $P ^{(k+1)}=\frac{d ^{2}}{d ^{2}-1}\left(P ^{(k)}-\frac{2}{d+1}\frac{P ^{(k)} g ^{(k)}g ^{(k)\top } P ^{(k)}}{g ^{(k)\top }P ^{(k)}g ^{(k)}}\right)$

      $k=k+1$
    }

    \KwOutput{$x _{\text{best}},v _{\text{best}}$}
\end{algorithm}

\subsubsection{Problem with Equality Constraints}

The approximate Kolmogorov width SDP \eqref{eq: SDP problem} can be equivalently formulated as a convex program over symmetric matrices, subject to the trace constraint $\text{tr}(X)=d-k$ and the positive semidefinite bounds $\mathbf{0}\preceq X \preceq \mathbf{I}_{d}$:
\begin{equation}\label{eq: outer optimization problem}
\begin{aligned}
&\min\limits _{X}h _{S}(X),\\
& \text{subject to } \text{tr}(X)=d-k, \mathbf{0}\preceq X \preceq \mathbf{I}_{d}, X ^{\top } =X.
\end{aligned}
\end{equation}
To execute the ellipsoid method in the presence of affine equality constraints, we adopt the framework established by \cite{SHAH200185}. Consider a general convex optimization problem with both inequality and equality constraints:
\begin{equation}\label{eq: optimization problem with equality constraints}
\begin{aligned}
&\text{minimize } f(x),\\
&\text{subject to } f _{i}(x)<0, Ax=b,
\end{aligned}
\end{equation}
where $A \in \mathbb{R}^{\text{rank}\left(A\right)\times d}$ is a matrix with full row rank. Assuming the initial iterate $x ^{(0)}$ lies within the affine subspace defined by $Ax=b$, \cite{SHAH200185} derive the following projected update rules for the center $x ^{(k)}$ and the shape matrix $P ^{(k)}$:
\begin{equation}\label{eq: update of the ellipsoid method with linear constraints}
  \begin{aligned}
  & r ^{(k)}:=\frac{\left(P ^{(k)}-P ^{(k)}A ^{\top } \left(AP ^{(k)}A ^{\top } \right)^{-1}AP ^{(k)}\right)g ^{(k)}}{\sqrt{g ^{(k)\top } \left(P ^{(k)}-P ^{(k)}A ^{\top } \left(AP ^{(k)}A ^{\top } \right)^{-1}AP ^{(k)}\right)g ^{(k)}}},\\
  & x ^{(k+1)}=x ^{(k)}-\frac{1}{d+1}r ^{(k)},\\
  & P ^{(k+1)}=\frac{d ^{2}}{d ^{2}-1}\left(P ^{(k)}-\frac{2}{d+1}r ^{(k)}r ^{(k)\top }\right).
  \end{aligned}
\end{equation}
Following standard cutting-plane logic, $g ^{(k)}\in \partial f _{0}(x ^{(k)})$ is an objective cut if $x ^{(k)}$ is fully feasible. If $x ^{(k)}$ violates any constraint, $g ^{(k)}\in \partial f _{i}(x ^{(k)})$ functions as a feasibility cut, where $i$ identifies the first constraint satisfying $f _{i}(x ^{(k)})\ge 0$.

While the update equations in \eqref{eq: update of the ellipsoid method with linear constraints} appear algebraically intricate, they stem from a straightforward geometric reduction. Specifically, the affine constraint $Ax=b$ allows any feasible solution $x$ to be explicitly parameterized as $x=A ^{\top } \left(AA ^{\top } \right)^{-1}b+By$. Here, $y \in \mathbb{R}^{d-\text{rank}\left(A\right)}$ serves as the coordinate vector in the reduced space, and the columns of the matrix $B \in \mathbb{R}^{d \times (d-\text{rank}\left(A\right))}$ constitute a basis for the null space $\mathcal{N}(A)$. This transformation seamlessly projects the original problem onto an affine subspace of dimension $d _{A}:= d-\text{rank}\left(A\right)$. Because convexity is preserved under affine composition, the functions $f _{i}(\cdot ),i=0,1,\dots,m$ remain convex over this restricted domain. By absorbing the equality constraints directly into the definition of the search space, the problem reduces to a standard inequality-constrained framework. As a result, the theoretical guarantees of Theorem \ref{theorem: convergence guarantee of vanilla ellipsoid method} carry over directly to this projected setting.

\begin{corollary}[]\label{corollary: convergence guarantee of ellipsoid method for problem with equality constraints}
Consider the convex optimization problem \eqref{eq: optimization problem with equality constraints}, and define the projected ellipsoid $\tilde{E}^{(k)}:=E ^{(k)}\cap \left\{x \left\lvert\right. Ax=b\right\}$. The volumetric decay rate is strictly bounded by
\begin{equation*}
\frac{\text{Vol}(\tilde{E} ^{(k+1)})}{\text{Vol}(\tilde{E} ^{(k)})}=\frac{d ^{d _{A}}}{(d+1)^{\frac{d _{A}+1}{2}}\cdot (d-1)^{\frac{d _{A}-1}{2}}}<1,
\end{equation*}
where $\text{Vol}(\cdot )$ denotes the Lebesgue measure in the $d _{A}$-dimensional affine subspace. Furthermore, suppose that for some $\epsilon >0$, the $\epsilon $-suboptimal feasible region $\left\{x \in \mathbb{R}^{d} \left\lvert\right. f(x)-f(x ^{\star })\le \epsilon , f _{i}(x)<0, Ax=b\right\}$ possesses a strictly positive $d _{A}$-dimensional volume, where $x ^{\star }$ represents the constrained global minimizer. Letting $x ^{\dagger }_{(M)}:=\mathop{\arg \min}\limits _{1\le i\le M}f(x ^{(i)})$ denote the best feasible iterate observed up to step $M$, we obtain the exponential convergence guarantee
\begin{equation*}
f(x ^{\dagger }_{(M)})-f(x ^{\star })\le c \cdot \exp\left\{-\frac{M}{2d _{A}^{2}}\right\},
\end{equation*}
where $c>0$ is an absolute constant determined solely by $f _{0}$.
\end{corollary}

The pseudo code of the ellipsoid method when equality constraints present is shown in Algorithm \ref{algorithm: ellipsoid method with equality constraints}.

\begin{algorithm}[htbp]
    \caption{ellipsoid method with (equality) constraints}\label{algorithm: ellipsoid method with equality constraints}

    \KwInput{$d, f _{0},...,f _{m},x ^{(0)},P ^{(0)},\epsilon ,c ^{\prime }$}

    $k=0, x _{\text{best}}=\mathbf{0}, v _{\text{best}}=\infty $

    $M=c ^{\prime }d ^{2}\ln \left(\frac{1}{\epsilon }\right)$

    \While{$k<M$}{
      \If{$x ^{(k)}$ is feasible}{

        \If{$f _{0}(x ^{(k)})<v _{\text{best}}$}{
          $x _{\text{best}}=x ^{(k)}$

          $v _{\text{best}}=f _{0}(x ^{(k)})$
        }

        Select any $g ^{(k)}\in \partial f _{0}(x ^{(k)})$
      }
      \Else{
        Find the smallest $i$ such that $f _{i}(x ^{(k)})\ge 0$

        Select any $g ^{(k)}\in \partial f _{i}(x ^{(k)})$
      }
      
      Set $r ^{(k)}=\frac{\left(P ^{(k)}-P ^{(k)}A ^{\top } \left(AP ^{(k)}A ^{\top } \right)^{-1}AP ^{(k)}\right)g ^{(k)}}{\sqrt{g ^{(k)\top } \left(P ^{(k)}-P ^{(k)}A ^{\top } \left(AP ^{(k)}A ^{\top } \right)^{-1}AP ^{(k)}\right)g ^{(k)}}}$

      Update $x ^{(k+1)}=x ^{(k)}-\frac{1}{d+1}r ^{(k)}$

      Update $P ^{(k+1)}=\frac{d ^{2}}{d ^{2}-1}\left(P ^{(k)}-\frac{2}{d+1}r ^{(k)}r ^{(k)\top }\right)$

      $k=k+1$
    }

    \KwOutput{$x _{\text{best}},v _{\text{best}}$}
\end{algorithm}

\subsection{Ellipsoid Method with Approximate Subgradient}

A prerequisite of the classical ellipsoid method is the availability of exact subgradients for the functions $f _{i}$ ($i=0,1,\dots,m$) to accurately determine the cutting planes. In our setting, however, Lemma \ref{lemma: approximate subgradient} demonstrates that the subgradient information for $h _{S}(X)$ is limited to the rank-one matrix $\mathcal{O}(X)\mathcal{O}(X)^{\top }$, which merely acts as a $\kappa $-approximate subgradient for some constant $\kappa >1$. Geometrically, the ellipsoid method can be viewed as an elimination process: rather than directly isolating the true minimizer, each iteration systematically prunes the region $E ^{(k)}\cap (E ^{(k+1)})^{\complement }$ mathematically precluded from containing it.

Specifically, at iteration $k$, any point $x$ satisfying the strict halfspace inequality $\left\langle x-x ^{(k)}, g ^{(k)} \right\rangle >0$ yields an objective value strictly worse than $f(x ^{(k)})$, fundamentally disqualifying it as the minimizer. By substituting the exact $g ^{(k)}$ with the approximate oracle $\mathcal{O}(X ^{(k)})\mathcal{O}(X ^{(k)})^{\top }$ in the update step, we essentially compromise the exactness of this elimination. Consequently, the guarantee of global exact minimization is lost; instead, the algorithm converges to an approximate optimum $X _{M\text{-best}}:=\mathop{\arg \min}\limits _{X \in \left\{X ^{(1)},\dots,X ^{(M)}\right\}}h _{S}(X)$ such that $f(X ^{\star })\le f(X _{M\text{-best}})\lesssim \kappa f(X ^{\star })$. We rigorously establish this convergence behavior in the following theorem.

\begin{theorem}[]\label{theorem: convergence guarantee of approximate ellipsoid method}
Consider the semidefinite program \eqref{eq: SDP problem} for approximating the Kolmogorov widths. Substituting the exact subgradient $g(\cdot ) \in \partial h _{S}(\cdot )$ with the approximate oracle $\mathcal{O}(\cdot )\mathcal{O}(\cdot )^{\top } $ in Algorithm \ref{algorithm: ellipsoid method with equality constraints} yields the following strict volumetric decay:
\begin{equation*}
\frac{\text{Vol}(\tilde{E} ^{(k+1)})}{\text{Vol}(\tilde{E} ^{(k)})}=\frac{d ^{2 \tilde{d}}}{(d ^{2}+1)^{\frac{\tilde{d}+1}{2}}\cdot (d ^{2}-1)^{\frac{\tilde{d}-1}{2}}}<1.
\end{equation*}
Let $X _{M\text{-best}}:=\mathop{\arg \min}\limits _{X \in \left\{X ^{(1)},\dots,X ^{(M)}\right\}}h _{S}(X)$ denote the best feasible iterate generated after $M$ steps of Algorithm \ref{algorithm: ellipsoid method with equality constraints}. The suboptimality gap with respect to the true global minimum $X ^{\star }$ is bounded by
\begin{equation*}
h _{S}(X _{M\text{-best}})\le \kappa \left(h _{S}(X ^{\star }) + c \cdot \exp\left\{-\frac{M}{2 \tilde{d}^{2}}\right\}\right).
\end{equation*}
Consequently, to achieve a $\kappa $-approximate solution up to an additive precision $\epsilon >0$, such that
\begin{equation*}
h _{S}(X ^{\star })\le h _{S}(X _{M \text{-best}})\le \kappa \left(h _{S}(X ^{\star })+\epsilon \right),
\end{equation*}
the algorithm requires at most $M \le c ^{\prime }\tilde{d}^{2}\cdot \ln \left(\frac{1}{\epsilon }\right)$ iterations. Here, $\tilde{d}=\frac{(d-1)(d+2)}{2}$ represents the intrinsic dimension of the search space under the trace (consumes $1$ degree of freedom) and symmetry constraints (consume $\frac{d(d-1)}{2}$ degrees of freedom).
\end{theorem}

\begin{remark*}[]
  One might initially find the coexistence of the ambient dimension $d^2$ and the intrinsic dimension $\tilde{d}$ in the decay rate counterintuitive. This distinction naturally arises because our algorithm executes the matrix updates in the full ambient space $\mathbb{R}^{d \times d}$ (which introduces the $d^2$ terms in the update rules), whereas the effective search space —-- and consequently the Lebesgue measure of the feasible region —-- is confined to a $\tilde{d}$-dimensional affine subspace.
\end{remark*}

The rigorous proof of Theorem \ref{theorem: convergence guarantee of approximate ellipsoid method} is deferred to Appendix \ref{subsection: proof of the convergence property of the approximate ellipsoid method}. While Theorem \ref{theorem: convergence guarantee of approximate ellipsoid method} theoretically guarantees that the best historical iterate $X _{M \text{-best}}$ achieves the desired suboptimality bound, this specific matrix is practically unidentifiable since the exact objective function $h _{S}(X)$ is unknown. To circumvent this, we leverage the approximation oracle $\mathcal{O}(X)$ developed in Section \ref{section: linear approximation} to evaluate the iterates. Specifically, replacing the intractable objective $h _{S}(X ^{(i)})$ with its computable surrogate $\mathcal{O}(X ^{(i)})^{\top } X ^{(i)} \mathcal{O}(X ^{(i)})$ yields the fully empirical Approximate Ellipsoid Method, detailed in Algorithm \ref{algorithm: ellipsoid method for the sdp}. In Algorithm \ref{algorithm: ellipsoid method for the sdp}, the constraint functions $f _{1},f _{2}$ for the inequality constraints $\mathbf{0}\preceq X \preceq \mathbf{I}_{d}$ can be obtained from the singular vectors regarding the maximal/minimal eigenvectors of $X ^{(k)}$. Specifically, $X \preceq \mathbf{I}_{d}$ is equivalent to $\lambda _{\text{max}}(X)\le 1$, and a sub-gradient of $\lambda _{\text{max}}(X)$ is $v _{\text{max}}v _{\text{max}}^{\top } $, where $v _{\text{max}}$ is any eigenvector of $\lambda _{\text{max}}$; similarly, $\mathbf{X}\succeq \mathbf{0}$ is equivalent to $-\lambda _{\text{min}}(X)\le 0$, with sub-gradient $-v _{\text{min}}v _{\text{min}}^{\top } $. In Theorem \ref{theorem: convergence guarantee of approximate ellipsoid method}, the scaling constant $c ^{\prime }$ is chosen to be sufficiently large and depends solely on $\epsilon $ and $K$, we define the practically selected optimum as $X _{k}^{\dagger }:=\mathop{\arg \min}\limits _{X \in \left\{X ^{(1)},\dots,X ^{(M)}\right\}}\mathcal{O}(X)^{\top } X \mathcal{O}(X)$, which immediately yields the following corollary.

\begin{corollary}[]\label{corollary: property of X_k^dg}
  For the output $X _{k}^{\dagger }$ of Algorithm \ref{algorithm: ellipsoid method for the sdp}, we have 
  \begin{equation}
    \frac{1}{\kappa }h _{S}(X ^{\star })\le \mathcal{O}(X _{k}^{\dagger })^{\top } X _{k}^{\dagger }\mathcal{O}(X _{k}^{\dagger })\le \kappa h _{S}(X ^{\star }).
  \end{equation}
\end{corollary}

\begin{algorithm}[t]
    \caption{ellipsoid method for SDP \eqref{eq: SDP problem}}\label{algorithm: ellipsoid method for the sdp}

    \KwInput{$d$, constraints $\text{tr}(X)=d-k,\mathbf{0}\preceq X \preceq \mathbf{I}_{d},X=X ^{\top } $, $X ^{(0)},P ^{(0)},\epsilon ,c ^{\prime }$}

    Deduce $A, f _{1},\partial f _{1},\dots$ from the equality and inequality constraints

    $k=0, X _{k}^{\dagger }=\mathbf{0}, v _{\text{best}}=\infty $

    $P ^{(0)}=4(d-k)\mathbf{I}_{d}$

    $M=c ^{\prime }d ^{2}\ln \left(\frac{1}{\epsilon }\right)$

    \While{$k<M$}{
      \If{$X ^{(k)}$ is feasible}{

        \If{$\mathcal{O}(X ^{(k)})^{\top } X ^{(k)}\mathcal{O}(X ^{(k)})<v _{\text{best}}$}{
          $X _{k}^{\dagger }=X ^{(k)}$

          $v _{\text{best}}=\mathcal{O}(X ^{(k)})^{\top } X ^{(k)}\mathcal{O}(X ^{(k)})$
        }

        $g ^{(k)}=\mathcal{O}(X ^{(k)})\mathcal{O}(X ^{k})^{\top } $
      }
      \Else{
        Find the smallest $i$ such that $f _{i}(X ^{(k)})\ge 0$

        Select any $g ^{(k)}\in \partial f _{i}(X ^{(k)})$
      }
      
      Set $r ^{(k)}=\frac{\left(P ^{(k)}-P ^{(k)}A ^{\top } \left(AP ^{(k)}A ^{\top } \right)^{-1}AP ^{(k)}\right)g ^{(k)}}{\sqrt{g ^{(k)\top } \left(P ^{(k)}-P ^{(k)}A ^{\top } \left(AP ^{(k)}A ^{\top } \right)^{-1}AP ^{(k)}\right)g ^{(k)}}}$

      Update $X ^{(k+1)}=X ^{(k)}-\frac{1}{d+1}r ^{(k)}$

      Update $P ^{(k+1)}=\frac{d ^{2}}{d ^{2}-1}\left(P ^{(k)}-\frac{2}{d+1}r ^{(k)}r ^{(k)\top }\right)$

      $k=k+1$
    }

    \KwOutput{$X _{k}^{\dagger },v _{\text{best}}$}
\end{algorithm}

\begin{proof}
  The first inequality follows from \eqref{eq: property of the oracle} and the definition of $X ^{\star }$:
  \begin{equation*}
    \mathcal{O}(X _{k}^{\dagger })^{\top } X _{k}^{\dagger }\mathcal{O}(X _{k}^{\dagger })\ge \frac{1}{\kappa }h(X _{k}^{\dagger })\ge \frac{1}{\kappa }h _{S}(X ^{\star }).
  \end{equation*}
  The second inequality follows from the definition of $X _{k}^{\dagger }$, (unknown) $X _{M \text{-best}}$, $h _{S}$ and Theorem \ref{theorem: convergence guarantee of approximate ellipsoid method}:
  \begin{equation*}
    \mathcal{O}(X _{k}^{\dagger })^{\top } X _{k}^{\dagger }\mathcal{O}(X _{k}^{\dagger })\le \mathcal{O}(X _{M \text{-best}})^{\top } X _{M \text{-best}}\mathcal{O}(X _{M \text{-best}})\le h _{S}(X _{M \text{-best}})\le \kappa h _{S}(X ^{\star }).
  \end{equation*}
  The proof is completed.
\end{proof}

Moving forward, we denote $\tilde{h} _{S}(X _{k}^{\dagger }):=\mathcal{O}(X _{k}^{\dagger })^{\top } X _{k}^{\dagger }\mathcal{O}(X _{k}^{\dagger })$ as the resulting approximate optimal value of $X _{k}^{\dagger }$.

\section{Testing Procedures}\label{section: testing procedures}

\subsection{Approximate Optimal Dimension}\label{subsection: approximate optimal dimension}

Armed with the approximation frameworks developed for the inner maximization \eqref{eq: inner optimization problem} in Section \ref{section: linear approximation} and the outer minimization in Section \ref{section: approximate subgradient and ellipsoid method}, we are now prepared to bypass the computational intractability inherent in evaluating the exact Kolmogorov widths. Recall that the SDP \eqref{eq: SDP problem} acts as a rigorous convex relaxation of the original width computation, with $X _{k}^{\dagger }$ serving as its computationally efficient approximate solution. Consequently, for general balanced, \hyperref[def_type_2_condition]{type-$2$} and \hyperref[def_exact_2_convex_norm]{2-convex} sets where $D _{k}(K)$ remains analytically or computationally elusive, it is a natural progression to substitute the true width with the surrogate quantity $\sqrt{\tilde{h}_{S}(X _{k}^{\dagger })}$. Building upon this practical substitution, we formally define the \textit{approximate optimal dimension} in terms of $\tilde{h}_{S}(X _{k}^{\dagger })$ as follows.

\begin{definition}[Approximate optimal Dimension]\label{def_approximate_optimal_dimensions}
  Define $\tilde{D}_{k}(K):=\sqrt{\tilde{h}_{S}(X _{k}^{\dagger })}, k=1,2\dots$ as the approximate Kolmogorov $k$-width. Assume that the data $\tilde{\mathbf{Y}}$ are drawn from the distribution $\mathcal{P}\in \mathscr{P}$. Let the function $f(\cdot ,\cdot ):\mathbb{N}\times \mathscr{P}\mapsto \mathbb{R}^{+}$, then the approximate optimal dimension regarding $K$ and $f$ is defined as 
  \begin{equation*}
    \tilde{k} ^{\star }_{f}:=\max\limits \left\{j \left\lvert\right. j \ge 0,\tilde{D}_{j-1}(K)>f(j,\mathcal{P})\right\},
  \end{equation*}
  where similarly, we additionally define $\tilde{D}_{k}(K)=\infty $ for $k<0$ and $\tilde{D}_{k}(K)=0$ when $k>\dim (\mathcal{X})$.
\end{definition}

As previously discussed, the exact first and second optimal dimensions are fundamental for characterizing the minimax rate in the robust $\ell _{2}$-norm testing framework \eqref{eq: basic testing problem}. When transitioning to their approximate counterparts—the approximate first and second optimal dimensions, which are obtained by selecting the identical criterion $f(j,\mathcal{P})$—we establish the following property.

\begin{lemma}[]\label{lemma: property of approximate first and second optimal dimensions}
  Let $f _{1}(j,\mathcal{P})=\frac{j ^{1/4}}{\sqrt{N}}\sigma ,f _{2}(j,\mathcal{P})=\frac{j ^{1/4}\sqrt{\epsilon }}{N ^{1/4}}\sigma $. For the exact first and second optimal dimensions $k _{1}^{\star },k _{2}^{\star }$ and the approximate first and second optimal dimensions $\tilde{k}_{1}^{\star }:=\tilde{k}_{f _{1}}^{\star },\tilde{k}_{2}^{\star }:=\tilde{k}_{f _{2}}^{\star }$, we have the following inequalities.
  \begin{equation}\label{eq: inequalities between exact and approximate first and second optimal dimensions}
    \begin{aligned}
      & \tilde{k}_{1}^{\star }\le \min\limits \left\{\kappa ^{2}(k _{1}^{\star }+1)-1,d\right\},\\ 
      & \tilde{k}_{2}^{\star }\le \min\limits \left\{\kappa ^{2}(k _{2}^{\star }+1)-1,d\right\}.
    \end{aligned}
  \end{equation}
\end{lemma}

\begin{remark*}[]
  Lemma \ref{lemma: property of approximate first and second optimal dimensions} states that the ratios between $\tilde{k}_{i}^{\star }, k _{i}^{\star },i \in \left\{1,2\right\}$ are bounded by at most $\kappa ^{2}$. Recall that $\kappa =\text{polylog}(d)$.
\end{remark*}

\begin{proof}
  The proof of Theorem \ref{lemma: property of approximate first and second optimal dimensions} is a direct application of the non-increasing property of the Kolmogorov widths. We only prove for $k _{1}^{\star }$ and $\tilde{k}_{1}^{\star }$, the same logic holds for $k _{2}^{\star }$ and $\tilde{k}_{2}^{\star }$ as well. By definition of $k _{1}^{\star }$, we have 
  \begin{equation*}
    D _{k _{1}^{\star }}(K)\le \frac{(k _{1}^{\star }+1)^{\frac{1}{4}}}{\sqrt{N}}\sigma .
  \end{equation*}
  Recall from Corollary \ref{corollary: property of X_k^dg} and Lemma \ref{lemma: approximate Kolmogorov width}, we know (omit the arbitrarily small constant $\epsilon $)
  \begin{equation*}
    \tilde{D}_{k}(K)=\sqrt{\tilde{h}_{S}(X _{k}^{\dagger })}\le \sqrt{\kappa h _{S}(X _{k}^{\star })}\le \sqrt{\kappa }D _{k}(K).
  \end{equation*}
  Consider $j=\kappa ^{2}(k _{1}^{\star }+1)\ge k _{1}^{\star }+1$, by the non-increasing property of the Kolmogorov widths, we know 
  \begin{equation*}
    \tilde{D}_{j}(K)\le \sqrt{\kappa }D _{j}(K)\le \sqrt{\kappa }D _{k _{1}^{\star }}(K)\le \sqrt{\kappa }\frac{(k _{1}^{\star }+1)^{\frac{1}{4}}}{\sqrt{N}}\sigma =\frac{j ^{\frac{1}{4}}}{\sqrt{N}}\sigma .
  \end{equation*}
  By definition, we know that $\tilde{k}_{1}^{\star }\le j-1=\kappa ^{2}(k _{1}^{\star }+1)-1$. $\tilde{k}_{1}^{\star }\le d$ is trivial. The proof is completed.
\end{proof}

\subsection{A Simple Case}\label{subsection: a simple case}

To build intuition, we first consider a simplified, unconstrained version of \eqref{eq: testing problem} where $N=1, \epsilon =0$, and $K=\mathbb{R}^{d}$. Upon observing a single random vector $Y \in \mathbb{R}^{d}$ from $\mathcal{N}(\mu ,\sigma ^{2}\mathbf{I}_{d})$, we wish to test the following hypotheses:
\begin{equation}\label{eq: basic testing problem}
\begin{aligned}
& H _{0}: \mu =0,\\
& H _{1}: \left\lVert \mu \right\rVert_{2}^{}\ge \rho .
\end{aligned}
\end{equation}
This formulation represents the canonical mean testing problem in the Gaussian sequence model, a subject exhaustively investigated in nonparametric statistics. Notably, for any fixed testing error level $\alpha $, \cite{ingster2003nonparametric} established that the fundamental minimax separation rate scales as $\rho _{\text{critical}}\asymp d ^{\frac{1}{4}}\sigma $. Formally, the following theorem holds:
\begin{theorem}[\cite{ingster2003nonparametric}]\label{theorem: results by Ingster with no constraint}
  When $K=\mathbb{R}^{d}$, for the testing problem \eqref{eq: basic testing problem}, denote $R(\rho ):=\min\limits _{T:\mathcal{X}\rightarrow \left\{0,1\right\}}\max\limits _{\mu  :\left\lVert \mu  \right\rVert_{2}^{}\ge \rho }R(\mu ,T)$ as the minimax risk function under the testing function $T$, where 
  \begin{equation*}
    R(\mu ,T):=\text{Type I error}+\text{Type II error}=\mathbb{P}_{\mathbf{0}}(T(Y)=1)+\mathbb{P}_{\mu }(T(Y)=0).
  \end{equation*}
  Define $\rho _{\text{critical}}:=d ^{\frac{1}{4}}\sigma $, then \\
  \begin{tabular}{ll}
    (1), & when $\rho \gtrsim \rho _{\text{critical}}$, we have $R(\rho )=o(1)$;\\ 
    (2), & when $\rho \lesssim \rho _{\text{critical}}$, we have $R(\rho )=\varTheta (1)$.
  \end{tabular}
\end{theorem}
When we operate under the constrained setting, i.e., $K \neq \mathbb{R}^{d}$, intuitively, the non-trivial additional information on the constraint should yield a refined rate of $\rho _{\text{critical}}$ compared to Theorem \ref{theorem: results by Ingster with no constraint}. \cite{li2026robustsignaldetectionquadratically} proved that for general QCO constraints, we have the following results
\begin{theorem}[\cite{li2026robustsignaldetectionquadratically}]\label{theorem: main results 1 of last paper}
  When $K$ is a QCO set, for the testing problem \eqref{eq: testing problem} with no corruption, define $\rho _{\text{critical}}:=\left(k _{1}^{\star }\right)^{\frac{1}{4}}\sigma $, then \\
  \begin{tabular}{ll}
    (1), & when $\rho \gtrsim \rho _{\text{critical}}$, we have $R(\rho )=o(1)$;\\ 
    (2), & when $\rho \lesssim \rho _{\text{critical}}$, we have $R(\rho )=\varTheta (1)$,
  \end{tabular}\\
  where $k _{1}^{\star }$ is the \textit{first optimal dimension}, i.e., selecting $f(j,\mathcal{P})=j ^{\frac{1}{4}}\sigma $ in Definition \ref{def_optimal_dimension}.
\end{theorem}

Theorem \ref{theorem: main results 1 of last paper} attains a sharper rate than Theorem \ref{theorem: results by Ingster with no constraint} since $k _{1}^{\star }\le d$. The core intuition is to project the observation $Y$ onto the lower-dimensional subspace $H _{P _{k _{1}^{\star }}^{\star }}$ spanned by the optimal projection $P _{k _{1}^{\star }}^{\star }$, thereby reformulating the testing problem within $H _{P _{k _{1}^{\star }}^{\star }}$. The resulting minimax separation rate in this reduced subspace is then linked to the original problem through the Kolmogorov widths $D _{k _{1}^{\star }}(K)$. Unfortunately, this oracle-type approach is largely unimplementable, as it strictly requires the knowledge of $k _{1}^{\star }$ and $P _{k _{1}^{\star }}^{\star }$, and consequently $D _{k}(K)$, which is intractable in general as highlighted in Section \ref{subsection: relaxed optimization problem}. We will include more discussions in Section \ref{subsection: polynomial algorithm}. Before that, let us first introduce the following  scenario as a preliminary.

Consider a random vector $Y \in \mathbb{R}^{d}$ where $Y=\mu +\zeta ,\mu \in \mathbb{R}^{d},\zeta \sim \mathcal{N}(\mathbf{0},\mathbf{\Sigma })$ with $\mathbf{\Sigma }\in \mathbb{R}^{d \times d}$ is the covariance matrix with the constraint $\text{tr}(\mathbf{\Sigma })=k \sigma ^{2}$ and $\mathbf{0}\preceq \mathbf{\Sigma }\preceq \sigma ^{2}\cdot \mathbf{I}_{d}$, where $k$ and $\sigma $ are known. We attempt to obtain the minimax rate for $\rho $ in the testing problem \eqref{eq: basic testing problem}. Define the following testing function of $\chi ^{2}$-type:
\begin{equation}\label{eq: chi_square test}
  T _{\chi }(Y)=\mathbf{1}_{E}, E :=\left\{Y \left\lvert\right. \left\lVert Y\right\rVert_{2}^{2}-k \sigma ^{2}\ge t\right\},
\end{equation}
where $t$ is some appropriate threshold to be determined. $\chi ^{2}$ test is known to be minimax optimal for the testing problem \eqref{eq: basic testing problem} when $\mathbf{\Sigma }=\mathbf{I}_{d}$. Specifically, we have the following results.

\begin{theorem}[]\label{theorem: simple main results 1}
  Given the assumptions above, for the problem \eqref{eq: basic testing problem}, if $\rho \gtrsim k ^{\frac{1}{4}}\sigma $, the test \eqref{eq: chi_square test} achieves uniform small error less than $\alpha $. Consequently, we have $\rho _{\text{critical}}\lesssim k ^{\frac{1}{4}}\sigma $ in this scenario. 
\end{theorem}
Theorem \ref{theorem: simple main results 1} broadens the classical results by relaxing the requirement of an identity covariance matrix; in fact, it operates effectively even without full knowledge of $\mathbf{\Sigma }$. Denote the $d$ eigenvalues of $\mathbf{\Sigma }$ by $\left\{\lambda _{1}\sigma ^{2},\dots,\lambda _{d}\sigma ^{2}\right\}$, sorted as $\lambda _{1}\ge \dots \ge \lambda _{d}$, under the trace constraint $\sum\limits _{i=1}^{d}\lambda _{i}=k$. For any fixed covariance matrix $\mathbf{\Sigma }$, applying the standard Le Cam's method yields a minimax lower bound for the testing problem \eqref{eq: basic testing problem} scaling as $\rho _{\text{critical}}\gtrsim \left(\sum\limits _{i=1}^{d}\lambda _{i}^{2}\right)^{\frac{1}{4}}\sigma $. Imposing the constraints on $\lambda _{i}$ for $i=1,\dots,d$, it follows algebraically that $\frac{k}{d ^{1/4}}\le \left(\sum\limits _{i=1}^{d}\lambda _{i}^{2}\right)^{\frac{1}{4}}\le k ^{\frac{1}{4}}$. Therefore, when $\mathbf{\Sigma }$ is unknown and the minimax framework takes the supremum over all feasible $\mathbf{\Sigma }$, the rate established in Theorem \ref{theorem: simple main results 1} is minimax optimal. Another consequence of this result is that the separation rate is purely characterized by the energy parameters of the covariance matrix (i.e., $k$ and $\sigma $), entirely bypassing the original dimension $d$.

Theorem \ref{theorem: simple main results 1} can be viewed as a natural consequence of the celebrated Hanson--Wright inequality. The proof is deferred to Appendix \ref{subsection: proof of simple main results 1}.

\subsection{Robust Testing}\label{subsection: robust testing}

In this section, we consider a more realistic scenario. For now, let us assume that we have $N$ observations $\mathbf{Y}=\left\{Y _{1},\dots,Y _{N}\right\}$ contaminated by the adversary $\mathcal{C}$ from the \hyperref[def_epsilon_contamination_model]{strong $\epsilon $-contamination model}: $\mathbf{Y}=\mathcal{C}(\tilde{\mathbf{Y}})$. For such setting, \cite{li2026robustsignaldetectionquadratically} prove the following minimax lower bound when the constraint $K$ is a QCO set.

\begin{theorem}[\cite{li2026robustsignaldetectionquadratically}]\label{theorem: main lower bound of last paper}
    For the mean testing problem \eqref{eq: testing problem} with potential contamination from a strong adversary $\mathcal{C}$ and a prior constraint on the mean $\mu$ induced by a QCO set $K$, the following condition is necessary to ensure the existence of a valid test whose Type I and Type II errors are both uniformly below a prescribed constant $\alpha $
    \begin{equation}\label{eq: main theoretical lower bound}
        \rho ^{2}\gtrsim \sigma ^{2}\max\limits \left\{\frac{\sqrt{k _{1}^{\star }}}{N},\epsilon \sqrt{\frac{k _{2}^{\star }}{N}},\epsilon ^{2}\right\},
    \end{equation}
    where $N$ is the sample size, $\sigma ^{2}$ is the variance of the noise, and $\epsilon $ is the fraction of contamination. $k _{1}^{\star }$ and $k _{2}^{\star }$ are the previously defined first and second optimal dimensions.
\end{theorem}

Theorem \ref{theorem: main lower bound of last paper} demonstrates that the minimax lower bound is jointly determined by three components: the geometric properties of the set $K$, which are captured by the optimal dimensions $k _{1}^{\star }$ and $k _{2}^{\star }$; the characteristics of the original data, which are parameterized by $N$ and $\sigma $; and the corruption process, which is dictated by $k _{2}^{\star }$ and $\epsilon $. The contamination fraction $\epsilon $ intricately connects all three terms in the lower bound. As $\epsilon $ increases from $0$ to $c _{0}$ (assuming $c _{0}=\varTheta (1)$), each of the three terms in the max operator dominates the other two in turn, thereby forming three distinct phases of model behavior.

\subsubsection{Polynomial-Time Algorithm}\label{subsection: polynomial algorithm}

In this section, we construct a polynomial-time testing procedure centered around the surrogate matrix $A ^{\dagger }_{k}$. Specifically, the core intuition is to project the observations $\mathbf{Y}$ into some subspace $H _{P _{k}^{\star }}$ of $\mathbb{R}^{d}$ with intrinsic dimension $k \le d$  induced by $k$-dimensional optimal projection $P _{k}^{\star }$. We then conduct the deduction based on the subspace $H _{P _{k}^{\star }}$ as if we transfer the testing problem to $\mathbb{R}^{k}$ where there is no constraint. The resulting upper bound on $\left\lVert \mu \right\rVert_{2}^{}$ depends on two parts. (1), the price we have to pay for the distance between the original $\mu $ and the projected $P _{k}^{\star }\mu $, which is $\left\lVert \mu -P _{k}^{\star }\mu \right\rVert_{2}^{}$; (2), the fundamental requirement on $\left\lVert P _{k}^{\star }\mu \right\rVert_{2}^{}$ when we attempt to conduct the testing in $H _{P _{k}^{\star }}^{}$. The best upper bound possible is achieved by balancing the two parts above, and this is one of the motivations of the optimal dimension and projection. This approach depending on the balance of the two parts already appears in the recent work by \cite{li2026robustsignaldetectionquadratically}. However, in this work, as mentioned previously, the fundamental difference is that we do not assume the knowledge of the Kolmogorov widths, and consequently the exact optimal dimensions and projections are unknown, neither. We therefore replace the exact Kolmogorov $k$-width with the solution to its relaxed version in \eqref{eq: SDP problem}, and also replace the exact optimal dimensions $k _{1}^{\star },k _{2}^{\star }$ and projections $P _{1}^{\star },P _{2}^{\star }$ with $\tilde{k}_{1}^{\star },\tilde{k} _{2}^{\star }$ defined in Definition \ref{def_approximate_optimal_dimensions} and $A _{1}^{\dagger }:=\left(\mathbf{I}_{d}-X _{\tilde{k}_{1}^{\star }}^{\dagger }\right)^{\frac{1}{2}},A _{2}^{\dagger }:=\left(\mathbf{I}_{d}-X _{\tilde{k}_{2}^{\star }}^{\dagger }\right)^{\frac{1}{2}}$ where $X _{\tilde{k}_{1}^{\star }}^{\dagger },X _{\tilde{k}_{2}^{\star }}^{\dagger }$ are the outputs of Algorithm \ref{algorithm: ellipsoid method for the sdp} when setting $k=\tilde{k}_{1}^{\star }$ and $k=\tilde{k}_{2}^{\star }$, respectively. Another issue derived from such approximate projections is that the covariance matrix of the projected samples is no longer identity. (When we use the exact projections, the covariance matrix can be thought as the identity as we might apply some known appropriate rotation to the projected samples.) We are already well prepared for this scenario, as is what Theorem \ref{theorem: simple main results 1} serves.

We also note here that while prior literature has explored robust testing and estimation for the unconstrained version of \eqref{eq: testing problem}, these classical procedures often suffer from restrictive assumptions --- requiring either a vanishing contamination fraction ($\epsilon \rightarrow 0$) or suboptimal conditions scaling with $N$ and $d$. Truly general and computationally tractable algorithms have only emerged recently. Notably, \cite{doi:10.1137/17M1126680} pioneered the first efficient robust estimation framework via algorithmic ``filtering'' in the Gaussian sequence model. For the robust testing counterpart, analogous filtering techniques were first successfully adapted by \cite{Narayanan2022PrivateHH, 10353143}, and later extended by \cite{li2026robustsignaldetectionquadratically}. This algorithmic framework to be leveraged is the basis of robust testing in the unconstrained setting (in $H _{A _{1}^{\dagger }}$ and $H _{A _{2}^{\dagger }}$). Below we review such framework. We start with the concept of \hyperref[def_omega_regularity_init]{$\omega $-regularity}, where a suit of weights is defined to adaptively ``filter'' the observations $\mathbf{Y}$.

\begin{definition}[$\omega $-regularity]\label{def_omega_regularity_init}
    Given a weight vector $\omega =(\omega _{1},\dots,\omega _{N})^{\top } $ and an integer $1 \le k \le d$, $\mathbf{Y}$ is said to be $(\epsilon ,\beta _{1},\beta _{2})$-regular if for all subsets $S \subset [N]$ with $\left|S\right|\le \epsilon N$, we have the following properties:\\ 
    \begin{tabular}{rl}
        (i), & $\left|\sum\limits_{i \in S}^{}\left\lVert Y _{i}\right\rVert_{2}^{2}-\left|S\right|k\right|\le c \beta _{1}$,\\ 
        (ii), & $\left|\left\lVert \sum\limits_{i \in S}^{}\sqrt{\omega }_{i} Y _{i}\right\rVert_{2}^{2}-\left\lVert \omega _{S}\right\rVert_{1}^{}k\right|\le c \beta _{2}$, and\\ 
        (iii), & $\left|\left\langle \sum\limits_{i \in S}^{}\sqrt{\omega }_{i}Y _{i}, \sum\limits_{j \in [N]}^{}\sqrt{\omega }_{j}Y _{j} \right\rangle-\left\lVert \omega _{S}\right\rVert_{1}^{} k\right|\le c \sqrt{N}\beta _{1}$,
    \end{tabular}

    where $\omega _{S} \in \mathbb{R}^{N}$ is the restriction of $\omega $ on the set $S$.
\end{definition}

Before proceeding, we note a slight abuse of notation: hereafter, $Y _{i}$ and $\mu $ should be implicitly understood as their approximately projected counterparts, $A _{k}^{\dagger }Y _{i}$ and $A _{k}^{\dagger }\mu $. Crucially, because the surrogate matrices $X _{k}^{\dagger }$ and $A _{k}^{\dagger }$ are constructed independently of the observation sets $\tilde{\mathbf{Y}}$ and $\mathbf{Y}$, the mutual independence among the projected observations $A _{k}^{\dagger }\tilde{Y}_{i}$ and $A _{k}^{\dagger }\tilde{Y}_{j}$ ($i \neq j$) is preserved. To avoid notational clutter, we drop the explicit projection operator $A _{k}^{\dagger }$ and simply write $\tilde{\mathbf{Y}}, \mathbf{Y}, Y _{i}$, and $\mu $ in the remaining text. We also assume $\sigma =1$ in this section for the conciseness of the notation without loss of generality. When $\sigma \neq 1$, we only need to scale all the results by $\sigma $.

As a baseline, consider the uncontaminated setting where $Y _{i}=\tilde{Y}_{i}$ for all $i \in [N]$. Under this clean-data regime, we naturally adopt the uniform weighting scheme $\omega =\mathbf{1}_{[N]}$ across the entire dataset. By standard concentration inequalities, it can be shown that this uncorrupted empirical distribution inherently satisfies the $(\epsilon ,\beta _{1},\beta _{2})$-regularity condition, with $\beta _{1}$ and $\beta _{2}$ explicitly given by:
\begin{equation}\label{eq: specified beta_1 and beta_2}
    \begin{aligned}
        &\beta _{1}=\epsilon N \left(\sqrt{k}+\sqrt{N}\left\lVert \mu \right\rVert_{2}^{}\right)\sqrt{\ln \left(\frac{N}{\alpha }\right)}+\epsilon N \ln \left(\frac{N}{\alpha }\right)+\epsilon N \sqrt{N}\left\lVert \mu \right\rVert_{2}^{2},\\
        &\beta _{2}=\epsilon N \sqrt{\epsilon Nk \ln \left(\frac{1}{\epsilon }\right)}+(\epsilon N)^{2}\ln \left(\frac{1}{\epsilon }\right)+\left\lVert \mu \right\rVert_{2}^{}(\epsilon N)^{2}\sqrt{\ln \left(\frac{1}{\epsilon }\right)}+\left\lVert \mu \right\rVert_{2}^{2}(\epsilon N)^{2}.
    \end{aligned}
\end{equation}

Consequently, the parameters $\beta _{1}$ and $\beta _{2}$ naturally motivate the formulation of a $\chi ^{2}$-type test statistic. Nevertheless, such deterministic regularity is invariably violated in the contaminated setting by the malicious adversary $\mathcal{C}$. To actively enforce the required $(\epsilon ,\beta _{1},\beta _{2})$-regularity condition on the empirical distribution, we propose the following three-step algorithmic framework.

\begin{algorithm*}[htbp]
    \caption{Pre-filtering.}\label{algorithm: prefiltering}
    Set $\gamma _{1}=c\left[\sqrt{k\ln \left(\frac{N}{\alpha }\right)}+\ln \left(\frac{N}{\alpha }\right)\right]$, $\texttt{count}=0,\texttt{i}=0$

    \While {$\texttt{i} < N$}{
        \If {$\left|\left\lVert Y _{i}\right\rVert_{2}^{2}-k\right|> \gamma _{1}$}{

            $\texttt{count}=\texttt{count}+1$

            \If {$\texttt{count}>\epsilon N$}{
                \Return \texttt{None}
            } 
            Delete $Y _{i}$ from $\mathbf{Y}$      
        }
        $\texttt{i}=\texttt{i}+1$
    }

    \Return $\mathbf{Y}$
\end{algorithm*}

Algorithm \ref{algorithm: prefiltering} is designed to explicitly enforce the first condition of \hyperref[def_omega_regularity_init]{$\omega $-regularity}. The threshold parameter $\gamma _{1}$ is carefully chosen based on the (sub)-Gaussian concentration of the uncorrupted data. Specifically, for clean samples, the concentration inequality $\left|\left\lVert \tilde{Y} _{i}\right\rVert_{2}^{2}-k\right|\le \gamma _{1}$ holds across all $i=1,2,\dots,N$ with probability at least $1-\alpha $. Because the adversarial contamination budget is strictly limited to $\epsilon $, observing strictly more than $\epsilon N$ violations of this bound provides sufficient statistical evidence to confidently reject the null hypothesis and conclude $H _{1}$. We formalize this pre-filtering logic in the following lemma.

\begin{lemma}[]\label{lemma: guarantee of omega regularity 1}
    After the execution of Algorithm \ref{algorithm: prefiltering}, we have for any $i \in [N]$,
    \begin{equation*}
        \left|\left\lVert Y _{i}\right\rVert_{2}^{2}-k ^{\star }\right| \le \gamma _{1}.
    \end{equation*}
    Moreover, under $H _{0}$, the algorithm will terminate without rejection with probability at least $1-\alpha $.
\end{lemma}

For the second condition, we separate the classical scenario when $k<N$ and the high-dimensional scenario when $k \ge N$. In the following text, let $D(\omega )=\text{diag}\{\omega \}\in \mathbb{R}^{N \times N}$.

\textbf{Classical scenario ($k < N$)}: select $\gamma _{2}$ as:
\begin{equation}\label{eq: definition of gamma_2}
    \gamma _{2}:=c \left[\sqrt{Nk}+\sqrt{N\ln \left(\frac{1}{\alpha }\right)}+\ln \left(\frac{1}{\alpha }\right)+\epsilon N \ln \left(\frac{1}{\epsilon }\right)\right],
\end{equation}
where $c$ is a sufficiently large but universal constant. Denote the value of $\omega $ in the $t$-th iteration as $\omega ^{(t)}$. Define $\tau _{i}=\left\langle v, Y _{i} \right\rangle ^{2}\mathbf{1}_{\left\{\omega _{i}>0\right\}}$, where $v$ is the unit singular vector associated with $\lambda =\left\lVert \mathbf{Y} ^{\top } D(\omega )\mathbf{Y}-N \mathbf{I}_{k}\right\rVert_{2}^{}$. The update policy for $\omega $ is as
\begin{equation}\label{eq: update of omega}
    \omega _{i}^{(t+1)}=\left\{
    \begin{tabular}{ll}
        $\left(1-\frac{\tau _{i}^{(t)}}{\tau _{1}^{(t)}}\right)\omega _{i}^{(t)}$ & if $i \le I$,\\
        $\omega _{i}^{(t)}$& if $i>I$.
    \end{tabular}
    \right.
\end{equation}
See the detailed filtering process for the second condition of \hyperref[def_omega_regularity_init]{$\omega $-regularity} under classical scenario in Algorithm \ref{algorithm n > k}.

\begin{algorithm}[htbp]
    \caption{Sample filtering when $N>k$.}\label{algorithm n > k}
    Set $\gamma _{2}$ as \eqref{eq: definition of gamma_2}, and $\lambda =\left\lVert \mathbf{Y} ^{\top } D(\omega )\mathbf{Y}-N \mathbf{I}_{k}\right\rVert_{2}^{}$. ($\omega $ is initialized as $\mathbf{1}$.)

    \While {$\lambda \ge \gamma _{2}$}{
        Set $v$ to be the unit singular vector associated with $\lambda $

        Compute $\tau _{i}=\left\langle v, Y _{i} \right\rangle ^{2}\mathbf{1}_{\left\{\omega _{i}>0\right\}}$ for $1 \le i \le N$

        Sort $\left\{\tau _{i}\right\}_{i=1}^{N}$ according to the decreasing order

        Set $I$ be the smallest index such that $\sum\limits_{i=1}^{I}\omega _{i}\ge 2\epsilon N$

        Update $w$ according to \eqref{eq: update of omega}

        \If {$\left\lVert \omega \right\rVert_{1}^{}<N(1-2\epsilon )$} {

            \Return \texttt{None}
        }

        Set $\lambda =\left\lVert \mathbf{Y} ^{\top } D(\omega )\mathbf{Y}-N \mathbf{I}_{k}\right\rVert_{2}^{}$
    }

    \Return $\omega $
\end{algorithm}

Algorithm \ref{algorithm n > k} guarantees the second condition of \hyperref[def_omega_regularity_init]{$\omega $-regularity} under the classical scenario (i.e., $N>k$) within polynomial-time complexity. In particular, we have the following lemma.

\begin{lemma}[]\label{lemma: guarantee of omega regularity 2 1}
    The following facts hold for Algorithm \ref{algorithm n > k}:\\
    \hspace*{0.5em}(i), it terminates within finite rounds,\\ 
    \hspace*{0.5em}(ii), under $H _{0}$, it does not output rejection (\texttt{None}) with probability greater than $1-\alpha $;\\
    \hspace*{0.5em}(iii), once finished without rejection, for any $S \subset [N]$ with $\left|S\right|\le \epsilon N$, we have $\left\lVert \sum\limits_{i \in S}^{}\sqrt{\omega _{i}}Y _{i}\right\rVert_{2}^{2}\le 2 \gamma _{2}\epsilon N$ with probability higher than $1-\alpha $.
\end{lemma}

\textbf{High-dimensional scenario ($k \geq N$)}: select $\gamma _{3}$ as:
\begin{equation}\label{eq: definition of gamma_3}
    \gamma _{3}=c \left(\sqrt{Nk}+\sqrt{k \ln \left(\frac{1}{\alpha }\right)}+\ln \left(\frac{1}{\alpha }\right)+\epsilon N \ln \left(\frac{1}{\epsilon }\right)\right),
\end{equation}
where again $c$ is some sufficiently large but universal constant. Define $\tau _{i}=\frac{v _{i}^{2}}{\omega _{i}}\mathbf{1}_{\left\{\omega _{i}>0\right\}}$ where $v$ is the unit singular vector associated with $\lambda =\left\lVert \sqrt{D(\omega )}\mathbf{Y}\mathbf{Y}^{\top }\sqrt{D(\omega )}-kD(\omega )\right\rVert_{2}^{}$. The update policy for $\omega $ in the $t$-th iteration is as:
\begin{equation}\label{eq: update of omega 2}
    \omega _{i}^{(t+1)}=\left(1-\frac{\tau _{i}}{\max\limits _{i}\tau _{i}}\right)\omega _{i}^{(t)}.
\end{equation}

See the detailed filtering process for the second condition of \hyperref[def_omega_regularity_init]{$\omega $-regularity} under high-dimensional scenario in Algorithm \ref{algorithm n <= k}.

\begin{algorithm}[htbp]
    \caption{Sample filtering when $N \le k$.}\label{algorithm n <= k}
    Set $\gamma _{3}$ as \eqref{eq: definition of gamma_3}, and $\lambda =\left\lVert \sqrt{D(\omega)}\mathbf{Y}\mathbf{Y} ^{\top } \sqrt{D(\omega)}-kD(\omega )\right\rVert_{2}^{}$. ($\omega $ is initialized as $\mathbf{1}$.)

    \While {$\lambda \ge \gamma _{3}$}{
        Set $v$ to be the unit singular vector associated with $\lambda $

        Compute $\tau _{i}=\frac{v _{i}^{2}}{\omega _{i}}\mathbf{1}_{\left\{w _{i}>0\right\}}$

        Update $w$ according to \eqref{eq: update of omega 2}

        \If {$\left\lVert \omega \right\rVert_{1}^{}<N(1-6\epsilon )$} {

            \Return \texttt{None}
        }

        Set $\lambda =\left\lVert \sqrt{D(\omega)}\mathbf{Y}\mathbf{Y} ^{\top } \sqrt{D(\omega)}-kD(\omega )\right\rVert_{2}^{}$
    }

    \Return $\omega $
\end{algorithm}

similar to Algorithm \ref{algorithm n > k}, the following lemma establishes that Algorithm \ref{algorithm n <= k} indeed enforces the second condition of \hyperref[def_omega_regularity_init]{$\omega $-regularity} under the high-dimensional scenario (i.e., $N \le k$).

\begin{lemma}[]\label{lemma: guarantee of omega regularity 2 2}
    The following facts hold for Algorithm \ref{algorithm n <= k}:\\
    \hspace*{0.5em}(i), it terminates within finite rounds;\\ 
    \hspace*{0.5em}(ii), under $H _{0}$, it does not output rejection (\texttt{None}) with probability greater than $1-\alpha $;\\
    \hspace*{0.5em}(iii), once finished without rejection, for any $S \subset [N]$ with $\left|S\right|\le \epsilon N$, we have 
    \begin{equation*}
        \left|\left\lVert \sum\limits_{i \in S}^{}\sqrt{\omega _{i}}Y _{i}\right\rVert_{2}^{2}-k\left\lVert \omega _{S}\right\rVert_{1}^{} \right|\le c \epsilon N \gamma _{3}
    \end{equation*}
    with probability greater than $1-\alpha $.
\end{lemma}

We note here that the techniques involving the interaction with the maximal eigenvalue and the corresponding eigenvector of either $\mathbf{Y}^{\top } D(\omega )\mathbf{Y}-N \mathbf{I}_{k}$ or $\sqrt{D(\omega )}\mathbf{Y}\mathbf{Y}^{\top } \sqrt{D(\omega )}-kD(\omega )$ in Algorithm \ref{algorithm n > k} and \ref{algorithm n <= k} first appeared in \cite{JMLR:v10:klivans09a} and was later adapted to the robust statistics by \cite{doi:10.1137/17M1126680,10.5555/3305381.3305485}.

Finally, the third condition of \hyperref[def_omega_regularity_init]{$\omega $-regularity} is guaranteed by the weight filtering algorithm, presented below. 

\begin{algorithm}[htbp]
    \caption{Weight filtering.}\label{algorithm: weight filtering}
    Compute $\tau _{i}=\left|\left\langle \sqrt{\omega _{i}}Y _{i}, \sum\limits_{j=1}^{N}\sqrt{\omega _{j}}Y _{j} \right\rangle-k\omega _{i}\right|\mathbf{1}_{\left\{\omega _{i}>0\right\}}, 1 \le i \le N$

    Sort $\tau _{i}$ by desreasing order and find the indices $\left\{i _{1},i _{2},\dots,i _{\epsilon N}\right\}$ corresponding to the first $\epsilon N$ maximal $\tau _{i}$

    Set $\omega _{i _{1}},\omega _{i _{2}},\dots,\omega _{i _{\epsilon N}}$ to zero

    \Return $\omega $
\end{algorithm}

The input weights $\omega $ in Algorithm \ref{algorithm: weight filtering} is the output of either Algorithm \ref{algorithm n > k} or Algorithm \ref{algorithm n <= k} according to the relationship between $N$ and $k$. Algorithm \ref{algorithm: weight filtering} enforces the third condition of \hyperref[def_omega_regularity_init]{$\omega $-regularity}, as established in the following lemma.

\begin{lemma}[]\label{lemma: guarantee of omega regularity 3}
    For the weights $\omega $ outputted by Algorithm \ref{algorithm: weight filtering} and any $S \subset [N]$ with $\left|S\right|\le \epsilon N$, we have 
    \begin{equation*}
        \left|\sum\limits_{i \in S}^{}\sum\limits_{j=1}^{N}\sqrt{\omega _{i}\omega _{j}}\left\langle Y _{i}, Y _{j} \right\rangle-k \left\lVert \omega _{S}\right\rVert_{1}^{}\right|\le c(\sqrt{N}\beta _{1}+\beta _{2}+\epsilon N \gamma ),
    \end{equation*}
    where $\beta _{1}$ and $\beta _{2}$ are defined in \eqref{eq: specified beta_1 and beta_2}, $\gamma $ is equal to either $\gamma _{2}$ or $\gamma _{3}$ accordingly.
\end{lemma}

As outlined above, the sequential execution of Algorithms \ref{algorithm: prefiltering}, \ref{algorithm n > k} (or \ref{algorithm n <= k}), and \ref{algorithm: weight filtering} allows us to adaptively refine the data weights $\omega $ in response to the contaminated observations $\mathbf{Y}$. This adaptive filtering strictly enforces the \hyperref[def_omega_regularity_init]{$\omega $-regularity} condition with respect to the predefined parameters $\beta _{1}$ and $\beta _{2}$. Crucially, endowed with this enforced regularity, the adversarially corrupted dataset $\mathbf{Y}$ is compelled to statistically mimic the concentration behavior of the uncorrupted samples $\tilde{\mathbf{Y}}$. We formalize this fundamental property in the following lemma.

\begin{lemma}\label{lemma: property of the filtered weights}
    Let $\omega $ be the output weights after executing Algorithm \ref{algorithm: prefiltering}, \ref{algorithm n > k} (or \ref{algorithm n <= k}), and \ref{algorithm: weight filtering} without rejection (otherwise, we simply reject $H _{0}$), and $\beta _{1},\beta _{2}$ seleced as in \eqref{eq: specified beta_1 and beta_2}. Then, with high probability we have 
    \begin{equation}\label{eq: property of the filtered weights}
      \left|\left\lVert \sum\limits_{i=1}^{N}\sqrt{\omega _{i}}Y _{i}\right\rVert_{2}^{2}-k \left\lVert \omega \right\rVert_{1}^{}-N ^{2}\left\lVert \mu \right\rVert_{2}^{2}\right|\le c \left(\sqrt{N}\beta _{1}+\beta _{2}+\epsilon N \gamma +\left(N ^{\frac{3}{2}}\left\lVert \mu \right\rVert_{2}^{}+\sqrt{Nk}\right)\sqrt{\ln \left(\frac{1}{\alpha }\right)}+\ln \left(\frac{1}{\alpha }\right)\right).
    \end{equation}
\end{lemma}

On the LHS of the preceding inequality, the term $k \left\lVert \omega \right\rVert_{1}^{}$ is a known quantity upon completion of the filtering algorithms. To guarantee a strictly detectable separation between $H _{0}$ and $H _{1}$, we require the signal term $N ^{2}\left\lVert \mu \right\rVert_{2}^{2}$ to asymptotically dominate all residual terms on the RHS. This bounding strategy closely mirrors the approach developed for the \hyperref[section: theoretical algorithm]{theoretical algorithm}. Through straightforward algebraic manipulations, we formally establish the following equivalence.
\begin{theorem}[]\label{theorem: raw main upper bound 2}
  Assume the same conditions as in Theorem \ref{theorem: raw main upper bound 1}, for any $A _{k}^{\dagger },1 \le k \le d$, if
  \begin{equation}\label{eq: raw required condition for the main upper bound 2}
   \left\lVert A _{k}^{\dagger }\mu \right\rVert_{2}^{2}\gtrsim \sigma ^{2}\max\limits \left\{\frac{\epsilon \ln \left(\frac{N}{\alpha }\right)}{\sqrt{N}},\epsilon ^{2}\ln \left(\frac{N}{\alpha }\right),\sqrt{\frac{\epsilon ^{2}k \ln \left(\frac{N}{\alpha }\right)}{N}},\frac{\sqrt{k}\ln \left(\frac{1}{\alpha }\right)}{N}\right\}:=P _{\text{raw}}^{2}(\epsilon ,K,N,\sigma ^{2},k),
\end{equation}
we are able to test \eqref{eq: testing problem} with Type I and Type II errors uniformly less than $\alpha $.
\end{theorem}

Since
\begin{equation*}
  \left\lVert A ^{\dagger }_{k}\mu \right\rVert_{2}^{}=\sqrt{\mu ^{\top } (A _{k}^{\dagger })^{2}\mu }=\sqrt{\mu ^{\top } \left(\mathbf{I}_{d}-X _{k}^{\dagger }\right)\mu }=\sqrt{\left\lVert \mu \right\rVert_{2}^{2}-\mu ^{\top } X _{k}^{\dagger }\mu }\ge \sqrt{\left\lVert \mu \right\rVert_{2}^{2}-\tilde{h}_{S}(X _{k}^{\dagger })},
\end{equation*}
we come to the following sufficient condition.

\begin{corollary}[]\label{corollay: main upper bound 2}
  Assume the same assumptions as in Theorem \ref{corollary: main upper bound 1}. For any $1 \le k \le d $, if
  \begin{equation}\label{eq: raw main upper bound 2}
    \left\lVert \mu \right\rVert_{2}^{2}\gtrsim \sigma ^{2}\left[\tilde{h}_{S}(X _{k}^{\dagger })+P _{\text{raw}}^{2}(\epsilon ,K,N,\sigma ^{2},k)\right],
  \end{equation}
  where $\tilde{k}_{1}^{\star },\tilde{k}_{2}^{\star }$ are defined as in Definition \ref{def_approximate_optimal_dimensions}, then we are able to test the problem \eqref{eq: testing problem} with uniformly small Type I and Type II errors. Moreover, the testing procedures can be completed within polynomial-logarithmic time of the parameters $\left(N,d,\frac{1}{\epsilon },\frac{1}{\alpha }\right)$ for any type-$2$, exactly $2$-convex constraint $K$.
\end{corollary}

Since \eqref{eq: raw main upper bound 2} holds for any $1 \le k \le d$, we are free to optimize over $k$ to obtain the sharpest bound:
\begin{equation}\label{eq: main upper bound 2}
  \begin{aligned}
    \left\lVert \mu \right\rVert_{2}^{2}&\gtrsim \min\limits _{1 \le k \le d}\sigma ^{2}\left[\tilde{h}_{S}(X _{k}^{\dagger })+P _{\text{raw}}(\epsilon ,K,N,\sigma ^{2},k)\right]\\ 
    &=\max\limits \left\{\frac{\epsilon \ln \left(\frac{N}{\alpha }\right)}{\sqrt{N}},\epsilon ^{2}\ln \left(\frac{N}{\alpha }\right),\sqrt{\frac{\epsilon ^{2}\min\limits \left\{\tilde{k}_{1}^{\star },\tilde{k}_{2}^{\star }\right\}\ln \left(\frac{N}{\alpha }\right)}{N}},\frac{\sqrt{\min\limits \left\{\tilde{k}_{1}^{\star },\tilde{k}_{2}^{\star }\right\}}\ln \left(\frac{1}{\alpha }\right)}{N}\right\}.
  \end{aligned}
\end{equation}

\begin{proof}
  Since we are allowed to optimize over $k$, we select $\tilde{k}_{1}^{\star }$. From the definition of $\tilde{k}_{1}^{\star }$, we know that $\tilde{h}_{S}(X _{\tilde{k}_{1}^{\star }}^{\dagger })=\tilde{D}_{\tilde{k}_{1}^{\star }}^{2}(K)\le \frac{\sqrt{\tilde{k}_{1}^{\star }}}{N}\sigma ^{2}$, which is dominated by the term $\frac{\sqrt{\tilde{k}_{1}^{\star }}\ln (1/\alpha )}{N}\sigma ^{2}$ in the maximal operator. Therefore, we can safely remove $\tilde{h}_{S}(X _{\tilde{k}_{1}^{\star }}^{\dagger })$ in this case. The same logic also holds for $\tilde{k}_{2}^{\star }$. The obtained two upper bounds are both valid for the problem \eqref{eq: testing problem}, hence we select the smaller one, which leads to \eqref{eq: main upper bound 2}. The proof is completed.
\end{proof}

When \eqref{eq: main upper bound 2} is satisfied, let $\tilde{k}^{\star }=\min\limits \left\{\tilde{k}_{1}^{\star },\tilde{k}_{2}^{\star }\right\}$ and $A_{\tilde{k}^{\star }}^{\dagger }$ be the corresponding approximate optimal projection. To distinguish between $H _{0}$ and $H _{1}$, we focus on the following weighted $\chi ^{2}$-type testing event
\begin{equation*}
  \left\{\left|\left\lVert \sum\limits_{i=1}^{N}\sqrt{\omega _{i}}A _{\tilde{k}^{\star }}^{\dagger }Y _{i}\right\rVert_{2}^{2}-\tilde{k}^{\star } \left\lVert \omega \right\rVert_{1}^{}\right|\ge c _{2} N ^{2}P _{\text{raw}}(\epsilon ,K,N,\sigma ^{2},\tilde{k}^{\star })\right\}.
\end{equation*}
We reject $H _{0}$ when the event is true and vice versa.

We benchmark our upper bound \eqref{eq: raw main upper bound 2} against the theoretical guarantees established in \cite{li2026robustsignaldetectionquadratically}. In an idealized oracle setting where the exact Kolmogorov widths and their corresponding optimal projections are known a priori, the approximate dimensions $\tilde{k}_{1}^{\star }$ and $\tilde{k}_{2}^{\star }$ can be perfectly substituted with the exact optimal dimensions $k _{1}^{\star }$ and $k _{2}^{\star }$. This substitution yields a sharper upper bound that matches the minimax lower bound \eqref{eq: main theoretical lower bound} up to logarithmic factors (specifically, $\ln N$, $\ln \left(\frac{1}{\epsilon }\right)$, and $\ln \left(\frac{1}{\alpha }\right)$) under QCO constraints. Remarkably, even in our practical setting where these Kolmogorov widths are inherently unknown, Lemma \ref{lemma: property of approximate first and second optimal dimensions} ensures that the fully empirical upper bound \eqref{eq: raw main upper bound 2} incurs merely an additional $\text{polylog}(d)$ multiplicative factor.

The proofs of Lemmas \ref{lemma: guarantee of omega regularity 1}, \ref{lemma: guarantee of omega regularity 2 1}, \ref{lemma: guarantee of omega regularity 2 2}, \ref{lemma: guarantee of omega regularity 3}, and \ref{lemma: property of the filtered weights}, which the rationale of the filtering method, share substantial overlap with the analysis in \cite{li2026robustsignaldetectionquadratically}. We therefore refer the interested reader to Appendix C.6 of \cite{li2026robustsignaldetectionquadratically} for the full derivations. Nevertheless, our framework requires a non-trivial adaptation to accommodate the non-identity covariance structure of the approximately projected observations $A _{k}^{\dagger }\tilde{Y}_{i}$. (Had $A _{k}^{\dagger }$ been an exact orthogonal projection, this covariance would trivially reduce to the identity within a rotated subspace). To rigorously circumvent this difficulty, we establish Lemma \ref{lemma: bound on the max eigenvalue of the covariance matrix} and Lemma \ref{lemma: bounds on the l2-operator norm of X^TX and XX^T} to govern the $\ell _{2}$-operator norms of the sample covariance and Gram matrices, respectively. The proofs of these two lemmas are deferred to Appendix \ref{subsection: necessary concentration inequality for the main upper bound 2}. The pseudo code of the polynomial-time algorithm is shown in the following.

\begin{algorithm}[htbp]
    \caption{unconditional polynomial-time algorithm.}\label{algorithm: polynomial algorithm}
    
    \Import $\mathbf{Y}, N,K,\sigma =1$, Subalgorithm \ref{algorithm: prefiltering}, \ref{algorithm n > k}, \ref{algorithm n <= k}, \ref{algorithm: weight filtering}

    \vspace*{1em}

    $\mathbf{Y}=\mathbf{Y}/\sigma $

    Compute $\tilde{k} _{1}^{\star }, A _{1}^{\dagger }, \tilde{k} _{2}^{\star }, A _{2}^{\dagger }$ from $N,K,\sigma =1,\epsilon $ as in Definition \ref{def_approximate_optimal_dimensions} and Algorithm \ref{algorithm: ellipsoid method for the sdp}

    \If{$\tilde{k} _{1}^{\star }\le \tilde{k} _{2}^{\star }$}{
        $\tilde{k} ^{\star }= \tilde{k} _{1}^{\star }, A ^{\dagger }=A _{1}^{\dagger }$
    }

    \Else{
        $\tilde{k} ^{\star }=\tilde{k} _{2}^{\star }, A ^{\dagger }=A _{2}^{\dagger }$
    }

    $\mathbf{Y}=\mathbf{Y}A ^{\dagger }$

    Execute Algorithm \ref{algorithm: prefiltering}, and record the returning value $\texttt{R} _{1}$

    \If {$\texttt{R} _{1}$ is \texttt{None}}{
        Reject $H _{0}$
        
        \Return
    }

    \Else {
        $\mathbf{Y}=\texttt{R} _{1}$
    }

    Set $\omega =\mathbf{1}_{[\tilde{k} ^{\star }]}$

    \If {$N>\tilde{k} ^{\star }$}{
        Execute Algorithm \ref{algorithm n > k}, and record the returning value $\texttt{R}_{2}$
    }
    \Else {
        Execute Algorithm \ref{algorithm n <= k}, and record the returning value $\texttt{R} _{2}$
    }
    \If {$\texttt{R}_{2}$ is \texttt{None}} {
        Reject $H _{0}$

        \Return
    }

    \Else {
        $\omega =\texttt{R} _{2}$
    }

    Execute Algorithm \ref{algorithm: weight filtering}, and record the returning value $\texttt{R} _{3}$

    $\omega =\texttt{R} _{3}$

    \If {$\left|\left\lVert \sum\limits_{i=1}^{N}\sqrt{\omega _{i}}Y _{i}\right\rVert_{2}^{2}-\tilde{k} ^{\star } \left\lVert \omega \right\rVert_{1}^{}\right|\ge c _{2} N ^{2}P(\epsilon ,\tilde{k} ^{\star },N,1)$}{
        Reject $H _{0}$
    }

    \Else {
        Accept $H _{0}$
    }

    \Return
\end{algorithm}

\section{Discussion}\label{section: discussion}

In this paper, we establish a polynomial-time algorithmic framework for robust signal detection, resolving the testing problem \eqref{eq: testing problem} with rigorous control over Type I and Type II errors. Our procedure accommodates any balanced, type-$2$, and $2$-convex constraint set $K$, thereby generalizing the QCO constraints analyzed in recent literature. The computational tractability of our framework hinges crucially on the SDP relaxation of the Kolmogorov widths \eqref{eq: SDP problem}, complemented by the approximation schemes detailed in \cite{10.1145/3406325.3451128} and Section \ref{section: approximate subgradient and ellipsoid method}. Remarkably, our empirical minimax upper bound derived in Section \ref{section: testing procedures} is near-optimal, incurring merely a $\text{polylog}(d)$ multiplicative penalty compared to the theoretical minimax rates achieved by \cite{li2026robustsignaldetectionquadratically} in the oracle QCO setting.

Looking forward, several compelling directions warrant further investigation. A natural extension is the robust $\ell _{p}$-norm testing problem, obtained by replacing the $\ell _{2}$-norm objective in \eqref{eq: testing problem} with an $\ell _{p}$-norm. For the regime $1 \le p<2$, \cite{li2026robustsignaldetectionquadratically} established that the minimax lower bound and corresponding upper bound successfully transfer under $p$-convex orthosymmetric (PCO) constraints. Nevertheless, given the inherent computational intractability of evaluating $\ell _{p}$-Kolmogorov widths, the existence of an efficient algorithmic counterpart remains unknown. The scenario where $p>2$ is fundamentally more challenging. Although the $\chi ^{p}$-test achieves minimax optimality in the unconstrained, corruption-free model \cite{ingster2003nonparametric}, the exact behavior of the minimax separation rates under the interplay of structural constraints and adversarial corruption is entirely open.

Another critical avenue is testing against a shifted null hypothesis, $H _{0}: \mu =\mu _{0}$ ($\mu _{0}\neq 0$). For clean data, \cite{wei2020local} proved that the minimax lower bound under ellipsoidal constraints is governed by localized Kolmogorov widths. A broader formulation extends this to composite null hypotheses, where $H _{0}: \left\lVert \mu \right\rVert_{2}^{}\le \rho _{0}$ is tested against $H _{1}:\left\lVert \mu \right\rVert_{2}^{}\ge \rho _{1}$. Relevant unconstrained and uncorrupted baselines can be found in \cite{kania2025testingimprecisehypotheses}, yet the robust constrained version stands as an open problem.

Finally, characterizing robust testing over more exotic geometries --- such as general symmetric, arbitrarily convex, or non-convex sets—remains a highly non-trivial frontier. Fully resolving this demands bridging two fundamental questions: (i) identifying the absolute information-theoretic limits (i.e., establishing minimax lower bounds and the existence of statistical upper bounds without computational constraints), and (ii) understanding the potential statistical-computational gap by designing efficient algorithms, recognizing that polynomial-time tractability may inherently require sub-optimal statistical conditions.

\newpage

\appendix

\bibliographystyle{plainnat}
\bibliography{biblio}

\newpage

\section{Theoretical Algorithm}\label{section: theoretical algorithm}

To provide a complete theoretical picture of the upper bounds presented in Section \ref{section: testing procedures}, this appendix details the fully information-theoretic testing procedure. Although this theoretical algorithm operates in exponential time—necessitating an exhaustive search over all possible subsets—it achieves a sharper statistical upper bound than the polynomial-time filtering developed in Section \ref{section: testing procedures}. At its core, this procedure is anchored by a classical $\chi ^{2}$-type test combined with the high-probability existence of a \textit{consistent subset} among the observations $\mathbf{Y}$. We formally introduce the definition of this consistency property below.

\begin{definition}[Consistent subset]\label{def_consistent_subset_init}
    A subset $S \subset \mathbf{Y}=\left\{Y _{1},\dots,Y _{N}\right\}$ is called a consistent subset regarding the contamination fraction $\epsilon $ and a test $\phi :2 ^{S}\rightarrow \left\{0,1\right\}$ if\\ 
    \hspace*{0.5em}(i), $\left|S\right|\ge (1-\epsilon )N$,\\ 
    \hspace*{0.5em}(ii), $\phi (S ^{\prime })=\phi (S)$ for any $S ^{\prime }\subset S$ with $\left|S ^{\prime }\right|\ge (1-2\epsilon )N$.
\end{definition}

The intuition driving the \hyperref[def_consistent_subset_init]{consistent subset} formulation is conceptually straightforward. Given the observations $\mathbf{Y}$ from adversarial contamination up to $\epsilon $ fraction, it is guaranteed to contain a subset of purely uncorrupted samples (specifically, $\mathbf{Y}_{[N]\backslash C}$, even though the corruption indices $C$ is unknown). Suppose $\mathbf{Y}_{S _{0}}$ forms a valid consistent subset for some index set $S _{0}\subset [N]$. It then becomes possible to link the test statistic evaluated on $\mathbf{Y}_{S _{0}}$ back to the unobservable original samples $\tilde{\mathbf{Y}}$. This exact correspondence is validated by demonstrating that, under appropriate conditions on $\rho $, $\tilde{\mathbf{Y}}$ inherently constitutes a consistent subset with high probability. As a result, evaluating the proposed $\chi ^{2}$-type test $\phi _{e}$ yields the crucial equivalence $\phi _{e}(\mathbf{Y}_{S _{0}})=\phi _{e}(\mathbf{Y}_{S _{0} \cap [N]\backslash C})=\phi _{e}(\mathbf{Y}_{[N]\backslash C})=\phi _{e}(\tilde{\mathbf{Y}})$. In essence, extracting and testing on $S _{0}$ allows us to completely bypass the adversarial corruption and perfectly replicate the oracle test on the authentic data $\tilde{\mathbf{Y}}$.

Recall the surrogate projection matrix $A _{k}^{\dagger }:=\left(\mathbf{I}_{d}-X _{k}^{\dagger }\right)^{\frac{1}{2}}$, and consider the following $\chi ^{2}$-type rejection region:
\begin{equation}\label{eq: definition of the indicator event}
E _{e}:=\left\{\left\lVert A _{k}^{\dagger }\sum\limits _{i \in S}^{}Y _{i}\right\rVert_{2}^{2}-k \left|S\right|\sigma ^{2}\ge c\left|S\right|^{2}E _{\text{raw}}^{2}(\epsilon ,K,N,\sigma ^{2},k)\right\}.
\end{equation}
Here, $S$ denotes a subset of the data (with a slight abuse of notation, $S$ interchangeably refers to the index set within $[N]$ and the corresponding observations in $\mathbf{Y}$), $c > 0$ is a carefully chosen absolute constant, and the threshold $E _{\text{raw}}^{2}(\epsilon ,K,N,\sigma ^{2},k)$ is defined as:
\begin{equation}\label{eq: raw required condition of the main upper bound 1}
E _{\text{raw}}^{2}(\epsilon ,K,N,\sigma ^{2},k):=\sigma ^{2}\max\limits \left\{\frac{\sqrt{k}}{N},\epsilon ^{2}\ln \left(\frac{1}{\epsilon }\right),\sqrt{\frac{\epsilon ^{2}\ln \left(\frac{1}{\epsilon }\right)k}{N}}\right\}.
\end{equation}
Evaluated on any subset $S$, the theoretical test $\phi _{e}$ rejects $H _{0}$ if and only if the event $E _{e}$ holds. Given this explicit formulation of the test $\phi _{e}$, the subsequent lemma establishes that the underlying clean dataset $\tilde{\mathbf{Y}}$ naturally forms a consistent subset with high probability.

\begin{theorem}[]\label{theorem: raw main upper bound 1}
  For any $1 \le k \le d$, if $\left\lVert A ^{\dagger }_{k}\mu \right\rVert_{2}^{}\gtrsim E _{\text{raw}}(\epsilon ,K,N,\sigma ^{2},k)$, then with probability $1-o(1)$, $\tilde{\mathbf{Y}}A _{k}^{\dagger }$ is consistent with respect to $\phi _{e}$ defined above. Moreover, for the original samples $\tilde{\mathbf{Y}}$, $\phi _{e}$ achieves $o(1)$ Type I and Type II errors uniformly over $\mu \in K$.
\end{theorem}

The framework of the proof of Theorem \ref{theorem: raw main upper bound 1} is essentially based on the proof of Theorem 3.8 of \cite{li2026robustsignaldetectionquadratically}, which is deferred to Appendix \ref{subsection: proof of the raw main upper bound 1}.

Following the discussions above, we know that when the selected $\mu $ and $A _{k}^{\dagger }$ satisfy the condition $\left\lVert A ^{\dagger }_{k}\mu \right\rVert_{2}^{}\gtrsim E _{\text{raw}}(\epsilon ,K,N,\sigma ^{2},k)$, $\phi _{e}$ provides us an approach to distinguish between $H _{0}$ or $H _{1}$ with uniformly small Type I and Type II errors despite the corruptions. similar to the argument for the polynomial-time algorithm, we have $\left\lVert A ^{\dagger }_{k}\mu \right\rVert_{2}^{}\ge \sqrt{\left\lVert \mu \right\rVert_{2}^{2}-\tilde{h}_{S}(X _{k}^{\dagger })}$. Therefore, it suffices that
\begin{equation*}
  \left\lVert \mu \right\rVert_{2}^{2}\gtrsim \tilde{h}_{S}(X _{k}^{\dagger })+E _{\text{raw}}^{2}(\epsilon ,K,N,\sigma ^{2},k).
\end{equation*}
Notice that the LHS is independent of $k$. Consequently, we can optimize over $k$ on the RHS for the best bound possible. This leads us to the main upper bound for the theoretical algorithm.

\begin{corollary}[Upper bound of the theoretical algorithm]\label{corollary: main upper bound 1}
  For the testing problem \eqref{eq: testing problem} with potential corruptions from a strong $\epsilon $-contamination adversary $\mathcal{C}$ and the prior constraint $K$ on the mean $\mu $, if
  \begin{equation}\label{eq: main upper bound 1}
    \left\lVert \mu \right\rVert_{2}^{}\gtrsim E ^{2}(\epsilon ,K,N,\sigma ^{2}):=\sigma ^{2}\max\limits \left\{\frac{\sqrt{\min\limits \left\{\tilde{k}_{1}^{\star },\tilde{k}_{2}^{\star }\right\}}}{N},\epsilon ^{2}\ln \left(\frac{1}{\epsilon }\right),\sqrt{\frac{\epsilon ^{2}\ln \left(\frac{1}{\epsilon }\right)\min\limits \left\{\tilde{k}_{1}^{\star },\tilde{k}_{2}^{\star }\right\}}{N}}\right\},
  \end{equation}
  then we are able to distinguish between $H _{0}$ and $H _{1}$ with error probability $o(1)$. Note that we do not require the computability of the Kolmogorov widths for the constraint $K$.
\end{corollary}

\begin{proof}
  Since we are allowed to optimize over $k$, we select $\tilde{k}_{1}^{\star }$. From the definition of $\tilde{k}_{1}^{\star }$, we know that $\tilde{h}_{S}(X _{\tilde{k}_{1}^{\star }}^{\dagger })=\tilde{D}_{\tilde{k}_{1}^{\star }}^{2}(K)\le \frac{\sqrt{\tilde{k}_{1}^{\star }}}{N}\sigma ^{2}$, which is one of the terms in the maximal operator. Therefore, we can safely remove $\tilde{h}_{S}(X _{k}^{\dagger })$ when $k= \tilde{k}_{1}^{\star }$. The same logic also holds for $\tilde{k}_{2}^{\star }$. The obtained two upper bounds are both valid for the problem \eqref{eq: basic testing problem}, hence we select the smaller one, which leads to \eqref{eq: main upper bound 1}. The proof is completed.
\end{proof}

Similar to the comparison in the efficient algorithm, $\tilde{k}_{1}^{\star },\tilde{k}_{2}^{\star }$ can be substituted with $k _{1}^{\star }$ and $k _{2}^{\star }$ if the Kolmogorov widths are prior knowledge. From Lemma \ref{lemma: property of approximate first and second optimal dimensions}, we know such replacement is $\mathcal{O}(\kappa ^{2})$ sub-optimal.

However, as mentioned, one drawback for such testing procedure is that it requires checking the results of $\phi _{e}$ on all subsets $\mathbf{Y}_{S}$ with $\left|\mathbf{Y}_{S}\right|\ge (1-2\epsilon )N$, whose amount can be as large as $\left(\frac{e}{2\epsilon }\right)^{2\epsilon N}$ (see Lemma \ref{lemma: combination number}) --- a number grows exponentially with $N$. Hence, the theoretical algorithm is computationally intractable when $N$ is even moderately large and only remains theoretically valuable for the existence of the upper bound.

\section{Deferred Proofs}\label{section: deferred proofs}

\subsection{Proof of Lemma \ref{lemma: QCO sets are type-2 and exactly 2-convex}}\label{subsection: proof of the property that QCO sets are type-2 and exactly 2-convex}

\begin{proof}
  Define $\tilde{K}^{2}=\left\{\tilde{x} \left\lvert\right. \tilde{x}\in \mathbb{R}^{d}, \tilde{x}=\eta \odot x, x \in K ^{2}\right\}$ as the symmetrized squared set of $K$, where $\otimes $ means entrywise multiplication. Since $K$ is orthosymmetric and quadratically convex, we know that $\tilde{K}^{2}$ is orthosymmetric and convex. Note that under such condition, the Minkowski gauges $\rho _{K}(\cdot )$ and $\rho _{\tilde{K}^{2}}(\cdot )$ is well known to be a norm, and satisfy the connection that $\rho _{K}^{2}(x)=\rho _{\tilde{K}^{2}}(x ^{2}),\forall x \in K$. Here again $x ^{2}$ is the entrywise square.

  For any $m \in \mathbb{N}^{+}$, consider $\sum\limits_{i=1}^{m}g _{i}x _{i}$. Its $j$-th coordinate is $\left(\sum\limits_{i=1}^{m}g _{i}x _{i}\right)_{j}=\sum\limits_{i=1}^{m}g _{i}x _{ij}\sim \mathcal{N}\left(0,\sum\limits_{i=1}^{m}x _{ij}^{2}\right)$. Therefore, we have $\left(\sum\limits_{i=1}^{m}g _{i}x _{i}\right)^{2}\sim \left(\sum\limits_{i=1}^{m}x _{ij}^{2}\right)\tilde{g}_{j}^{2}$, where $\tilde{g}_{j}^{2}\sim \chi _{1}^{2}$. Since $\tilde{K} ^{2}$ is convex and contains the origin, we know that it is star-shape, and consequently for any $r,s \in \mathbb{R}^{d}$, if $\left|r _{i}\right|\le \left|s _{i}\right|,1 \le i \le d$, we have $\rho _{\tilde{K}^{2}}(r)\le \rho _{\tilde{K}^{2}}(s)$. We have 
  \begin{equation*}
    \begin{aligned}
      \mathbb{E}\left[\rho _{K}^{2}\left(\sum\limits_{i=1}^{m}\epsilon _{i}x _{i}\right)\right]&=\mathbb{E}\left[\rho _{\tilde{K}^{2}}\left(\sum\limits_{i=1}^{m}\epsilon _{i}x _{i}\right)^{2}\right]\overset{\text{(i)}}{\le }\frac{\pi }{2}\mathbb{E}\left[\rho _{\tilde{K}^{2}}\left(\sum\limits_{i=1}^{m}g _{i}x _{i}\right)^{2}\right]\\ 
      &\le \frac{\pi }{2}\rho _{\tilde{K}^{2}}\left(\sum\limits_{i=1}^{m}x _{i}^{2}\right)\mathbb{E}\max\limits _{j}\tilde{g}_{j}^{2}\overset{\text{(ii)}}{\le } c \frac{\pi }{2}\ln d \sum\limits_{i=1}^{m}\rho _{\tilde{K}^{2}}(x _{i}^{2})\\ 
      &= c \frac{\pi }{2}\ln d \sum\limits_{i=1}^{m}\rho _{K}^{2}(x _{i}^{2}).
    \end{aligned}
  \end{equation*}
  Here (i) is from Lemma 4.5 of \cite{ledoux2013probability}, and (ii) is from the standard maximum inequality for $\chi _{1}^{2}$ random variables. Therefore, the type-$2$ constant for any QCO set is $T _{2}(K)\asymp \sqrt{\ln d}$.
  
  For the exact 2-convexity, we have:
  \begin{equation*}
    \rho _{K}^{2}\left(\sum\limits_{i=1}^{m}x _{i}^{2}\right)^{\frac{1}{2}}=\rho _{\tilde{K}^{2}}\left(\sum\limits_{i=1}^{m}x _{i}^{2}\right)\le \sum\limits_{i=1}^{m}\rho _{\tilde{K}^{2}}(x _{i}^{2})=\sum\limits_{i=1}^{m}\rho _{K}^{2}(x _{i}).
  \end{equation*}
  Taking the square root on both sides proves the exact 2-convexity. The proof is completed.
\end{proof}

\subsection{Proof of Theorem \ref{theorem: convergence guarantee of approximate ellipsoid method}}\label{subsection: proof of the convergence property of the approximate ellipsoid method}
\begin{proof}
  The proof is derived from Lemma \ref{lemma: approximate subgradient} and Corollary \ref{corollary: convergence guarantee of ellipsoid method for problem with equality constraints}. First, the optimization problem \eqref{eq: SDP problem} is indeed a convex problem, where the object function $h _{S}$ is convex, and all the constraints are linear. Second, the feasible region for the problem is bounded. This can be verified by calculating the distance between any two feasible matrices $X _{1},X _{2}$: $\left\lVert X _{1}-X _{2}\right\rVert_{F}^{2}\le 2 \left\lVert X _{1}\right\rVert_{F}^{2}+2\left\lVert X _{2}\right\rVert_{F}^{2}$. Assuming the singular value decomposition of $X _{i}$ is $X _{i}=U _{i}\Lambda _{i}U _{i}^{\top } $, where $i \in \left\{1,2\right\}$, $\Lambda _{i}=\text{diag}\{\lambda _{i,1},\dots,\lambda _{i,d}\}$ and $\lambda _{i,j}$ is the $j$-th singular value of $X _{i}$, then we have $\left\lVert X _{i}\right\rVert_{F}^{2}=\text{tr}(X _{i}^{2})=\text{tr}(U _{i}\Lambda _{i}^{2}U _{i}^{\top })=\sum\limits_{j=1}^{d}\lambda _{i,j}^{2}\overset{\text{(i)}}{\le }\sum\limits_{j=1}^{d}\lambda _{i,j}=d-k$, where (i) is from the condition that $\mathbf{0}\preceq X \preceq \mathbf{I}_{d}$. Therefore, for any initial feasible point $X ^{(0)}$ (for example, $X ^{(0)}=\text{diag}\{\underbrace{1,\dots,1}_{d-k},0,\dots,0\}$), we can set $E ^{(0)}=\left\{X \left\lvert\right. \left\lVert X-X ^{(0)}\right\rVert_{F}^{}\le 2 \sqrt{d-k}\right\}$, and $E ^{(0)}$ is guaranteed to contain the minimizer $X ^{\star }$.

  Consider the $\epsilon $-tolerance set
  \begin{equation*}
    E _{\epsilon }:=\left\{X \left\lvert\right. X \in \mathbb{R}^{d \times d}, h _{S}(X)-h _{S}(X ^{\star })\le \epsilon \right\}.
  \end{equation*}
  The constraints of $X$ can be expressed as $\left\{X \left\lvert\right. AX=B\right\}$, where $A,B$ can be determined from the trace constraint $\sum\limits_{i=1}^{d}X _{ii}=k$ and the symmetry constraint $X _{ij}=X _{ji},\forall i \neq j$. The rank of the equality constraints is $1+\frac{d(d-1)}{2}$, and therefore the dimension of $\mathcal{N}(A)=d ^{2}-\left(1+\frac{d(d-1)}{2}\right)=\frac{(d-1)(d+2)}{2}$. It is also not hard to check that for any $\epsilon >0$, the intersection set $E _{\epsilon }\cap \left\{X \left\lvert\right. AX=B\right\}$
  has positive volume when considering in the $\mathbb{R}^{\frac{(d-1)(d+2)}{2}}$-dimensional manifold  determined by $A,B$ in $\mathbb{R}^{d \times d}$.

  Finally, in the $k$-th iteration:
  \begin{itemize}
    \item If $X ^{(k)}$ is not feasible, then any point in $E ^{(k+1)\complement}\cap E ^{(k)}$ is not feasible as well and will be excluded in the next iteration.
    \item If $X ^{(k)}$ is feasible, then for any point $X$ in $E ^{(k+1)\complement}\cap E ^{(k)} \cap \left\{X \left\lvert\right. AX=B\right\}$, from Lemma \ref{lemma: approximate subgradient} we know that $h _{S}(X)\ge \frac{1}{\kappa }h _{S}(X ^{(k)})+\mathcal{O}(X ^{(k)})^{\top } (X-X ^{(k)})\mathcal{O}(X ^{(k)})\ge \frac{1}{\kappa }h _{S}(X ^{(k)})$.
  \end{itemize}
  From Corollary \ref{corollary: convergence guarantee of ellipsoid method for problem with equality constraints}, we know that the volume of $E ^{(k)}\cap \left\{X \left\lvert\right. AX=B\right\}$ decreases exponentially with the factor $\frac{d ^{2 \tilde{d}}}{(d ^{2}+1)^{(\tilde{d}+1)/2}\cdot (d ^{2}-1)^{(\tilde{d}-1)/2}}<1$. Therefore, within $M:=c ^{\prime }\tilde{d}^{2}\cdot \ln \left(\frac{1}{\epsilon }\right)$ iterations, we have $\text{Vol}(E ^{(M)}\cap \left\{X \left\lvert\right. AX=B\right\})\le \text{Vol}(E _{\epsilon }\cap \left\{X \left\lvert\right. AX=B\right\})$, which means that there exists a feasible point, denoted by $X _{\text{e}}\in E _{\epsilon }$, is excluded along the iterations. Assume that it is excluded in $i$-th iteration. Then from the analysis above, we know that 
  \begin{equation*}
    h _{S}(X ^{\star })\le h _{S}(X _{M\text{-best}})\le h _{S}(X ^{(i)})\le \kappa h _{S}(X _{\text{e}})\le \kappa \left(h _{S}(X ^{\star })+\epsilon \right).
  \end{equation*}
  The proof is completed.
\end{proof}

\subsection{Proof of Theorem \ref{theorem: simple main results 1}}\label{subsection: proof of simple main results 1}

We start with the following simple fact.

\begin{fact*}
  For a hypothesis testing problem $H _{0}:\theta \in \Theta _{0}, H _{1}: \theta \in \Theta _{1}$ and a test statistic $M(\mathbf{Y})$. If there exists a threshold $t$ independent with $\theta $ such that 
  \begin{equation*}
    \begin{aligned}
      & \sup\limits_{\theta \in \Theta _{0}}\mathbb{P}_{\theta }(M(\mathbf{Y})\le t)\ge 1-\alpha ,\\ 
      & \sup\limits_{\theta \in \Theta _{1}}\mathbb{P}_{\theta }(M(\mathbf{Y})> t)\ge 1-\alpha ,
    \end{aligned}
  \end{equation*}
  then the test $\mathbf{1}_{\left\{M(\mathbf{Y})\le t\right\}}$ is valid with uniform Type I and Type II errors less than $\alpha $. As an application, if for a quantity $v _{\theta }$ depending on $\theta $ such that $\mathbb{P}_{\theta }(\left|M(\mathbf{Y})-\mathbb{E}_{\theta }(M(\mathbf{Y}))\right|\le v _{\theta })\ge 1-\alpha $ holds uniformly for $\theta \in \Theta _{1}\cup \Theta _{1}$ and $\sup\limits_{\theta \in \Theta _{0}}\mathbb{E}_{\theta }(M(\mathbf{Y}))+v _{\theta }\le \inf\limits_{\theta \in \Theta _{1}}\mathbb{E}_{\theta }(M(\mathbf{Y}))-v _{\theta }$ (or $\inf\limits_{\theta \in \Theta _{0}}\mathbb{E}_{\theta }(M(\mathbf{Y}))-v _{\theta }\ge \sup\limits_{\theta \in \Theta _{1}}\mathbb{E}_{\theta }(M(\mathbf{Y}))+v _{\theta }$), then there exists a valid testing achieving Type I and Type II errors less than $\alpha $ uniformly for $\theta \in \Theta _{0}\cup \Theta _{1}$.
\end{fact*}

Before proceeding, we also briefly recall the celebrated Hanson--Wright inequality. This fundamental result provides sharp tail bounds for quadratic forms of sub-Gaussian random vectors, making it an indispensable tool for analyzing Gaussian observations with non-identity covariance structures. The original inequality was established by \cite{10.1214/aoms/1177693335}, while a modern and streamlined proof can be found in \cite{rudelson2013hansonwrightinequalitysubgaussianconcentration}.

\begin{lemma}[The Hanson-Wright inequality]\label{lemma: hw_inequality}
    Assume that $\mathbf{X}=(X _{1},\dots,X _{n})\in \mathbb{R}^{n}$ is a random vector with independent, zero-mean and sub-Gaussian coordinates. Let $A$ be an $n \times n$ matrix and $K=\max \limits _{i}\left\lVert X _{i}\right\rVert_{\psi _{2}}^{}$. Then for any $t \ge 0$, we have 
     \begin{equation}\label{eq: hw_inequality}
        \mathbb{P}\left(|X ^{\top } AX-\mathbb{E}X ^{\top } AX|>t\right)\le 2 \exp\left\{-c \min\limits \left\{\frac{t ^{2}}{K ^{4}\left\lVert A\right\rVert_{F}^{2}},\frac{t}{K ^{2}\left\lVert A\right\rVert_{2}^{}}\right\}\right\},
     \end{equation}
     where $\left\lVert \cdot \right\rVert_{F}^{}$ is the Frobenius norm of a matrix, and $\left\lVert \cdot \right\rVert_{2}^{}$ is the $\ell _{2}$ operator norm of a matrix.
\end{lemma}

We are now ready for the proof of Theorem \ref{theorem: simple main results 1}.
\begin{proof}
  We analyze the term $\left\lVert Y\right\rVert_{2}^{2}$ directly. Since $Y=\mu +\zeta $, we have 
  \begin{equation*}
    \begin{aligned}
      &\left\lVert Y\right\rVert_{2}^{2}=\left\lVert \mu \right\rVert_{2}^{2}+2 \underbrace{\mu ^{\top } \zeta }_{\text{I}}+\underbrace{\left\lVert \zeta \right\rVert_{2}^{2}}_{\text{II}},\\ 
      &\mathbb{E}\left\lVert Y\right\rVert_{2}^{2}=\left\lVert \mu \right\rVert_{2}^{2}+k \sigma ^{2}.
    \end{aligned}
  \end{equation*}
  We bound the variation of the both terms I and II. Note that since $\zeta \sim \mathcal{N}(0,\Sigma )$, we can represent it as $\zeta =\Sigma ^{\frac{1}{2}}Z$ for $Z \sim \mathcal{N}(0,\mathbf{I}_{d})$.

  \textbf{Bound on I.} We have $\mu ^{\top } \zeta =\mu ^{\top } \Sigma ^{\frac{1}{2}}Z$. By standard properties of multivariate normal distribution and sub-Gaussian random variables, we know $\mu ^{\top } \zeta \sim \mathcal{N}(0,\mu ^{\top } \Sigma \mu )$ and therefore 
  \begin{equation*}
    \mathbb{P}\left(\left|\mu ^{\top } \zeta \right|\ge t\right)\le 2 \exp\left\{-c \frac{t ^{2}}{\mu ^{\top } \Sigma \mu }\right\}\overset{\text{(1)}}{\le } 2 \exp\left\{-c \frac{t ^{2}}{\sigma ^{2}\left\lVert \mu \right\rVert_{2}^{2}}\right\},
  \end{equation*}
  where (1) is from the condition that $\mathbf{0}\preceq \Sigma \preceq \sigma ^{2}\mathbf{I}_{d}$ and $c$ is a universal constant. Therefore, for $t \gtrsim \sigma \left\lVert \mu \right\rVert_{2}^{}$, we have 
  \begin{equation*}
    \mathbb{P}\left(\left|\mu ^{\top } \zeta \right|\ge t\right)=\mathcal{O}(1).
  \end{equation*}

  \textbf{Bound on II.} Since $\left\lVert \zeta \right\rVert_{2}^{2}=Z ^{\top } \Sigma Z$, by the Hanson-Wright's inequality \eqref{eq: hw_inequality}, we know that
  \begin{equation*}
    \mathbb{P}\left(\left|Z ^{\top } \Sigma Z-k \sigma ^{2}\right|\ge t\right)\le 2 \exp\left\{-c \min\limits \left\{\frac{t ^{2}}{\left\lVert \Sigma \right\rVert_{F}^{2}},\frac{t}{\left\lVert \Sigma \right\rVert_{2}^{}}\right\}\right\}.
  \end{equation*}
  From the condition that $\mathbf{0}\preceq \Sigma \preceq \sigma ^{2}\mathbf{I}_{d}$ and $\text{tr}(\Sigma )=k \sigma ^{2}$, we have 
  \begin{equation*}
    \begin{aligned}
      & \left\lVert \Sigma \right\rVert_{F}^{2}=\sigma ^{4}\sum\limits_{i=1}^{d}\lambda _{i}^{2}\le \sigma ^{4}  \sum\limits_{i=1}^{d}\lambda _{i}=k \sigma ^{4},\\ 
      & \left\lVert \Sigma \right\rVert_{2}^{}=\max\limits \left\{\left|\lambda _{i}\right|\right\}=\lambda _{\text{max}}\le \sigma ^{2}.
    \end{aligned}
  \end{equation*}
  Plug in the inequalities above, we have
  \begin{equation*}
    \mathbb{P}\left(\left|Z ^{\top } \Sigma Z-k \sigma ^{2}\right|\ge t\right)\le 2 \exp\left\{-c \min\limits \left\{\frac{t ^{2}}{k \sigma ^{4}},\frac{t}{\sigma ^{2}}\right\}\right\},
  \end{equation*}
  where $c$ is again a universal constant. Consequently, for $t \gtrsim \sqrt{k}\sigma ^{2}$, we have 
  \begin{equation*}
    \mathbb{P}\left(\left|Z ^{\top } \Sigma Z-k \sigma ^{2}\right|\ge t\right)=\mathcal{O}(1).
  \end{equation*}
  Finally, to ensure the effectiveness of the test \eqref{eq: chi_square test}, it suffices that the variations are asymptotically dominated by the mean:
  \begin{equation*}
    \begin{aligned}
      &\sigma \left\lVert \mu \right\rVert_{2}^{}\lesssim \left\lVert \mu \right\rVert_{2}^{2},\\ 
      &\sqrt{k}\sigma ^{2}\lesssim \left\lVert \mu \right\rVert_{2}^{2},
    \end{aligned}
  \end{equation*}
  which is equivalent to requiring that $\left\lVert \mu \right\rVert_{2}^{}\gtrsim k ^{\frac{1}{4}}\sigma $. The proof is completed.
\end{proof}

\subsection{Proof of Theorem \ref{theorem: raw main upper bound 1}}\label{subsection: proof of the raw main upper bound 1}
\begin{proof}
In this section, we provide the underlying rationale for Theorem \ref{theorem: raw main upper bound 1}, which is also the basis of Corollary \ref{corollary: main upper bound 1}. The central strategy mirrors the one employed in Appendix \ref{subsection: proof of simple main results 1}: ensuring that the unknown variations are asymptotically dominated by the known expectations. The crucial difference here is that this technique must be integrated with the definition of the \hyperref[def_consistent_subset_init]{consistent subset}.

As a preliminary step, we recall a classical estimate for binomial coefficients. This bound will be essential for controlling the number of subsets of $[N]$ with cardinality exceeding $(1-2\epsilon )N$ in our subsequent analysis.
\begin{lemma}[]\label{lemma: combination number}
    For any $0 \le k \le n$, we have $\sum\limits_{i=0}^{k}\binom{n}{i}\le \left(\frac{en}{k}\right)^{k}$.
\end{lemma}

\begin{proof}
    By the Taylor's expansion of $e ^{k}$, we know $e ^{k}=\sum\limits_{i=0}^{\infty }\frac{k ^{i}}{i!}\ge \sum\limits_{i=0}^{k}\frac{k ^{i}}{i!}$. Therefore, it suffices to show that $\frac{k ^{k}}{n ^{k}}\binom{n}{i}\le \frac{k ^{i}}{i!}$ for any $0 \le i \le k$, and this is equivalent to showing that $\frac{k ^{k}}{n ^{k}}\prod\limits_{j=0}^{i-1}(n-j)\le k ^{i}$. However, the last statement can be directly obtained from the fact that $\frac{k ^{k-i}}{n ^{k-i}}\le 1 \le \frac{n ^{i}}{\prod\limits_{j=0}^{i-1}(n-j)}$.
\end{proof}

Consider $E _{e}$ as in \eqref{eq: definition of the indicator event} for $S=S ^{\prime }$, where $S ^{\prime }$ can be any subset of $[N]$ with $\left|S \right|\ge (1-2\epsilon )N$. We first focus on the authentic samples (i.e., let $\mathbf{Y}=\tilde{\mathbf{Y}}$), and analyze the square term. By centering the samples, we have 
\begin{equation}\label{eq: decomposition of projection}
    \left\lVert A _{k}^{\dagger }\sum\limits_{i \in S ^{\prime }}^{}\tilde{Y} _{i}\right\rVert_{2}^{2}=\underbrace{\left\lVert A _{k}^{\dagger }\sum\limits_{i \in S ^{\prime }}^{}(\tilde{Y} _{i}-\mu )\right\rVert_{2}^{2}}_{:=\text{I}}+\underbrace{\left|S ^{\prime }\right|^{2}\left\lVert A _{k}^{\dagger }\mu \right\rVert_{2}^{2}}_{:=\text{II}}+2 \underbrace{\left|S ^{\prime }\right|\sum\limits_{i \in S ^{\prime }}^{}\left[A _{k}^{\dagger }(\tilde{Y} _{i}-\mu )\right]^{\top } \mu }_{:=\text{III}}.
\end{equation}
\textbf{Bound on I.} By representing $S ^{\prime }=[N]\backslash \left([N]\backslash S ^{\prime }\right)$, we have 
\begin{equation*}
    \text{I}=\underbrace{\left\lVert A _{k}^{\dagger }\sum\limits_{i \in [N]}^{}(\tilde{Y} _{i}-\mu )\right\rVert_{2}^{2}}_{:=\text{I}_{1}}+\underbrace{\left\lVert A _{k}^{\dagger }\sum\limits_{i \in [N] \backslash S ^{\prime }}^{}(\tilde{Y} _{i}-\mu )\right\rVert_{2}^{2}}_{:=\text{I}_{2,S ^{\prime }}}-2 \underbrace{\sum\limits_{i \in [N]}^{}(\tilde{Y} _{i}-\mu )^{\top } \left(\mathbf{I}_{d}-X _{k}^{\star }\right)\sum\limits_{i \in [N]\backslash S ^{\prime }}^{}(\tilde{Y} _{i}-\mu )}_{:=\text{I}_{3,S ^{\prime }}}.
\end{equation*}
For $\text{I}_{1}$, the arguments of concentration is from the Hanson-Wright inequality \ref{eq: hw_inequality}. Since $\mathbb{E}\text{I} _{1}=kN \sigma ^{2}\left\lVert \mathbf{I}_{d}-X _{k}^{\dagger }\right\rVert_{F}^{2}\le k,\left\lVert \mathbf{I}_{d}-X _{k}^{\dagger }\right\rVert_{2}^{}\le 1$. We know that 
\begin{equation*}
    \mathbb{P}\left(\left|\text{I}_{1}-kN \sigma ^{2}\right|\ge t\right)\le 2 \exp\left\{-c \min\limits \left\{\frac{t ^{2}}{k N ^{2}\sigma ^{4}},\frac{t}{N\sigma ^{2}}\right\}\right\}.
\end{equation*}
By adjusting $t$ to control the probability less than $\alpha $, we obtain
\begin{equation*}
  t \ge c\sqrt{k}\sigma ^{2}N
\end{equation*}
for some constant $c$ that only depends on $\alpha $.\npar

For $\text{I}_{2,S ^{\prime }}$, we combine the Hanson--Wright inequality with the union bound on $S ^{\prime }$. For any fixed $S ^{\prime }$ with $\left|S ^{\prime }\right|\ge (1-2\epsilon )N$, by the Hanson--Wright inequality \ref{lemma: hw_inequality} and the same conditions of $\mathbf{I}_{d}-A _{k}^{\dagger }$ as in the bound of $\text{I}_{1}$, we have 
\begin{equation*}
    \mathbb{P}\left(\left|\text{I} _{2}-k\left|[N]\backslash S ^{\prime }\right|\sigma ^{2}\right| \ge t\right)\le 2\exp\left\{-c \min\limits \left\{\frac{t ^{2}}{k \left|[N]\backslash S ^{\prime }\right|^{2}\sigma ^{4}},\frac{t}{\left|[N]\backslash S ^{\prime }\right|\sigma ^{2}}\right\}\right\}.
\end{equation*}
Now consider all such possible $S ^{\prime }$, whose amount is at most $\sum\limits_{i=0}^{2\epsilon N}\binom{N}{i}=\left(\frac{e}{2\epsilon }\right)^{2\epsilon N}$ by Lemma \ref{lemma: combination number}. Using the union bound, we have $\text{I} _{2,S ^{\prime }}-k\left|[N]\backslash S ^{\prime }\right|\sigma ^{2}\ge t$ holds for some $S ^{\prime }$ with probability less than
\begin{equation*}
    \sum\limits_{S ^{\prime }}^{}\exp\left\{-c \min\limits \left\{\frac{t ^{2}}{k\left|[N]\backslash S ^{\prime }\right|^{2}\sigma ^{4}},\frac{t}{\left|[N]\backslash S ^{\prime }\right|\sigma ^{2}}\right\}\right\}\le \left(\frac{e}{2\epsilon }\right)^{2\epsilon N}\cdot \exp\left\{-c \min\limits \left\{\frac{t ^{2}}{k\epsilon ^{2}N^{2}\sigma ^{4}},\frac{t}{\epsilon N\sigma ^{2}}\right\}\right\}.
\end{equation*}
If we select $t _{\text{I}_{2}}$ to be
\begin{equation*}
    t _{\text{I}_{2}}=c\max\limits \left\{\sqrt{\left[\epsilon N \ln \left(\frac{e}{\epsilon }\right)+\ln \left(\frac{1}{\alpha }\right)\right]k\epsilon ^{2}N ^{2}\sigma ^{4}},\left[\epsilon N \ln \left(\frac{1}{\epsilon }\right)+\ln \left(\frac{1}{\alpha }\right) \right]\epsilon N \sigma ^{2}\right\}
\end{equation*}
then it can be verified that $\left|\text{I} _{2,S ^{\prime }}-k \left|[N]\backslash S ^{\prime }\right|\sigma ^{2}\right|\le  t _{\text{I}_{2}}$ holds for all $S ^{\prime }$ with probability greater than $1-\alpha $.\npar

For $\text{I}_{3,S ^{\prime }}$, first, observe that the joint distribution of $\sum\limits_{i \in [N]\backslash S ^{\prime }}^{}(\tilde{Y} _{i}-\mu )$ and $\sum\limits_{i \in [N]}^{}(\tilde{Y} _{i}-\mu )$ is a multivariate Gaussian
\begin{equation*}
    \left(\sum\limits_{i \in [N]\backslash S ^{\prime }}^{}(\tilde{Y} _{i}-\mu )^{\top } ,\sum\limits_{i \in [N]}^{}(\tilde{Y} _{i}-\mu )^{\top } \right)\sim \mathcal{N}\left(\mathbf{0},\begin{pmatrix}
        \left|[N]\backslash S ^{\prime }\right| \mathbf{I}_{d} & \left|[N]\backslash S ^{\prime }\right| \mathbf{I}_{d}\\ 
        \left|[N]\backslash S ^{\prime }\right| \mathbf{I}_{d} & N \mathbf{I}_{d}
    \end{pmatrix}\sigma ^{2}\right).
\end{equation*}
By the property of multivariate Gaussian distribution, conditional on the event $\left\{\sum\limits_{i \in [N]}^{}(\tilde{Y} _{i}-\mu )=y\right\}$, the conditional distribution of $\sum\limits_{i \in [N]\backslash S ^{\prime }}^{}(\tilde{Y} _{i}-\mu )$ is
\begin{equation*}
    \left[\sum\limits_{i \in [N]\backslash S ^{\prime }}^{}(\tilde{Y} _{i}-\mu )\left\lvert\right. \sum\limits_{i \in [N]}^{}(\tilde{Y} _{i}-\mu )=y \right]\sim \mathcal{N}\left(\frac{\left|[N]\backslash S ^{\prime }\right|}{N}y,\sigma ^{2}\left(\left|[N]\backslash S ^{\prime }\right|-\frac{\left|[N]\backslash S ^{\prime }\right|^{2}}{N}\right)\mathbf{I}_{d}\right).
\end{equation*}
From the maximal inequality that $\mathbb{P}\left(\max\limits _{1 \le i \le N}\tilde{Y} _{i}\ge \sqrt{2 \sigma ^{2}(\ln N+t)}\right)\le e ^{-t}$, and conditional on the event $\left\{\sum\limits_{i \in [N]}^{}(\tilde{Y} _{i}-\mu )=y\right\}$, we have the probability of the following event 
\begin{equation*}
    \left\{\text{I}_{3,S ^{\prime }}\ge \frac{\left|[N]\backslash S ^{\prime }\right|}{N}y ^{\top } \left(\mathbf{I}_{d}-X _{k}^{\dagger }\right) y+\sqrt{2y ^{\top } \left(\mathbf{I}_{d}-X _{k}^{\dagger }\right) y \left(\left|[N]\backslash S ^{\prime }\right|-\frac{\left|[N]\backslash S ^{\prime }\right|^{2}}{N}\right)\sigma ^{2}\cdot \left(2\epsilon N\ln \left(\frac{e}{2\epsilon }\right)+t\right)}\right\}
\end{equation*}
holding for some $S ^{\prime }\subset [N]$ less than $e ^{-t}$. Also, for any general random variables $X$, $Y$, and $a \in \mathbb{R}$, we have
\begin{equation*}
    \begin{aligned}
        \mathbb{P}\left(Y \ge a\right)&=\mathbb{E}\mathbf{1}_{\left\{Y \ge a\right\}}\\ 
        &=\mathbb{E}\left[\mathbf{1}_{\left\{Y \ge a\right\}}\cdot \left(\mathbf{1}_{\left\{f(X)\ge a\right\}}+\mathbf{1}_{\left\{f(X)<a\right\}}\right)\right]\\ 
        &=\mathbb{E}\left[\mathbb{E}\left[\mathbf{1}_{\left\{Y \ge a\right\}}\cdot \mathbf{1}_{\left\{f(X)\ge a\right\}}\left|\right.X\right]+\mathbb{E}\left[\mathbf{1}_{\left\{Y \ge a\right\}}\cdot \mathbf{1}_{\left\{f(X)<a\right\}}\left|\right.X\right]\right]\\ 
        &\le \mathbb{E}\left[\mathbb{E}\left[\mathbf{1}_{\left\{f(X)\ge a\right\}}\left|\right.X\right]+\mathbb{E}\left[\mathbf{1}_{\left\{Y \ge f(X)\right\}}\cdot \mathbf{1}_{\left\{f(X)<a\right\}}\left|\right.X\right]\right]\\ 
        &\le \mathbb{P}\left(f(X)\ge a\right)+\mathbb{E}\left[\mathbb{P}\left(Y \ge f(X)\left|\right.X\right)\right].
    \end{aligned}
\end{equation*}
Now observe that $\left|[N]\backslash S ^{\prime }\right|\le 2\epsilon N$ for all $S ^{\prime }$, and take $Y=\max\limits _{S ^{\prime }}\left\{\text{I}_{3,S ^{\prime }}\right\}$, $X=\sum\limits_{i \in [N]}^{}(\tilde{Y} _{i}-\mu )$, $f(y)=g \left(y ^{\top } \left(\mathbf{I}_{d}-X _{k}^{\dagger }\right)y\right)$ where $g(y)=\frac{\left|[N]\backslash S ^{\prime }\right|}{N}y+\sqrt{4y \cdot \epsilon N \left(\epsilon N \ln \left(\frac{e}{\epsilon }\right)+t\right)}$ and $a=g(Nk ^{\star }\sigma ^{2}+\sqrt{\frac{4k ^{\star }}{\alpha }}N \sigma ^{2})$. By the argument above and the fact that $g(\cdot )$ is an increasing function on the positive real line, we know
\begin{equation*}
    \mathbb{P}\left(f(X)\ge a\right)=\mathbb{P}\left(y ^{\top } \left(\mathbf{I}_{d}-X _{k}^{\dagger }\right) y \ge Nk \sigma ^{2}+\sqrt{\frac{4k }{\alpha }}N \sigma ^{2}\right)\le \frac{\alpha }{2} .
\end{equation*}
From the maximal inequality above, we also have $\mathbb{E}\left[\mathbb{P}\left(Y \ge f(X)\left|\right.X\right)\right]\le \mathbb{E}e ^{-t}=e ^{-t}$. Taking $t=\ln \left(\frac{2}{\alpha }\right)$, and combining the results above, we conclude that with probability greater than $1-\alpha $ the following event holds for all $S ^{\prime }$.
\begin{equation*}
    \left\{\text{I}_{3,S ^{\prime }}\le \left|[N]\backslash S ^{\prime }\right|k \sigma ^{2}+\left|[N]\backslash S ^{\prime }\right|\sqrt{\frac{4k }{\alpha }}\sigma ^{2}+\sqrt{8 \left(Nk \sigma ^{2}+\sqrt{\frac{4k }{\alpha }}N \sigma ^{2}\right)\epsilon N \left(2\epsilon N \ln \left(\frac{e}{2\epsilon }\right)+\ln \left(\frac{2}{\alpha }\right)\right)\sigma ^{2}}\right\}.
\end{equation*}
By the symmetrical nature of $\text{I}_{3,S ^{\prime }}$, same technique and argument can also be applied to the probability of the lower tail of $\text{I}_{3,S ^{\prime }}$.

\textbf{Bound on II.} It is part of the expectation, and there is no randomness.

\textbf{Bound on III.} By spliting $S ^{\prime }=[N]\backslash \left([N]\backslash S ^{\prime }\right)$, we have
\begin{equation*}
    \text{III}=\underbrace{\left|S ^{\prime }\right|\sum\limits_{i \in [N]}^{}(\tilde{Y} _{i}-\mu )^{\top }A _{k}^{\dagger }\mu }_{\text{III}_{1,S ^{\prime }}}-\underbrace{\left|S ^{\prime }\right|\sum\limits_{i \in [N]\backslash S ^{\prime }}^{}(\tilde{Y} _{i}-\mu )^{\top } A _{k}^{\dagger }\mu }_{\text{III}_{2,S ^{\prime }}}.
\end{equation*}
For $\text{III}_{1,S ^{\prime }}$, observe that $\text{III}_{1,S ^{\prime }}\sim \mathcal{N}\left(0,\left|S ^{\prime }\right|^{2}N \sigma ^{2}\left\lVert A _{k}^{\dagger }\mu \right\rVert_{2}^{2}\right)$. By the Chebyshev's inequality and the fact $\left|S ^{\prime }\right|\le N$, we have
\begin{equation*}
    \mathbb{P}\left(\left|\text{III}_{1, S ^{\prime }}\right|\le \sqrt{\frac{N ^{3}\sigma ^{2}\left\lVert P ^{\star }\mu \right\rVert_{2}^{2}}{\alpha }}\right)\ge 1-\alpha .
\end{equation*}
For $\text{III}_{2,S ^{\prime }}$, similarly, we have $\text{III}_{2,S ^{\prime }}\sim \mathcal{N}\left(0, \left|S ^{\prime }\right|^{2} \left|[N]\backslash S ^{\prime }\right|\left\lVert P ^{\star }\mu \right\rVert_{2}^{2}\sigma ^{2}\right)$ for each fixed $S ^{\prime }$. By the Chebyshev's inequality and the union bound over $S ^{\prime }$ (similar argument as before), we know that the following holds.
\begin{equation*}
    \mathbb{P}\left(\left|\text{III}_{2,S ^{\prime }}\right|\le \sqrt{2\epsilon N ^{3}\left\lVert A _{k}^{\dagger }\mu \right\rVert_{2}^{2}\sigma ^{2}}\cdot \sqrt{2\epsilon N \ln \left(\frac{e}{2\epsilon }\right)+\ln \left(\frac{1}{\alpha }\right)},\forall S ^{\prime }\subset [N]\right)\ge 1-\alpha .
\end{equation*}

In summary, for each unknown term of I and III, we have the following expectations and concentrations on the variations.

\makebox[2.5em][l]{$\text{I}_{1}$}: \parbox[t]{0.9\textwidth}{
    Expectation: $kN \sigma ^{2}$,\\ 
    Variation: $\sqrt{\frac{2k }{\alpha }}N \sigma ^{2}$,\\ 
    Techniques: the Hanson--Wright inequality.
}
\npar

\makebox[2.5em][l]{$\text{I}_{2,S ^{\prime }}$}: \parbox[t]{0.9\textwidth}{
  Expectation: $k\left|[N]\backslash S ^{\prime }\right|\sigma ^{2}$,\\ 
  Variation: $c\max\limits \left\{\sqrt{\left[\epsilon N \ln \left(\frac{e}{\epsilon }\right)+\ln \left(\frac{1}{\alpha }\right)\right]k\epsilon ^{2}N ^{2}\sigma ^{4}},\left[\epsilon N \ln \left(\frac{1}{\epsilon }\right)+\ln \left(\frac{1}{\alpha }\right) \right]\epsilon N \sigma ^{2}\right\}$,\\ 
  Techniques: Lemma \ref{lemma: combination number}, union bound over $S ^{\prime }$, and the Hanson--Wright inequality.
}
\npar

\makebox[2.5em][l]{$\text{I}_{3,S ^{\prime }}$}: \parbox[t]{0.9\textwidth}{
  Expectation: $k\left|[N]\backslash S ^{\prime }\right|\sigma ^{2}$,\\ 
  Variation: $\left[\left|[N]\backslash S ^{\prime }\right|\sqrt{\frac{4k}{\alpha }}+\sqrt{8 \left(Nk+\sqrt{\frac{4k}{\alpha }}N \right)\epsilon N \left(2\epsilon N \ln \left(\frac{e}{2\epsilon }\right)+\ln \left(\frac{2}{\alpha }\right)\right)}\right]\sigma ^{2} $,\\ 
  Techniques: conditional argument and (sub-)Gaussian maximal inequality.
}
\npar

\makebox[2.5em][l]{$\text{III}_{1,S ^{\prime }}$}: \parbox[t]{0.9\textwidth}{
  Expectation: $0$,\\ 
  Variation: $\sqrt{\frac{N ^{3}\sigma ^{2}\left\lVert A _{k}^{\dagger }\mu \right\rVert_{2}^{2}}{\alpha }}$,\\ 
  Techniques: Chebyshev's inequality.
}
\npar

\makebox[2.5em][l]{$\text{III}_{2,S ^{\prime }}$}: \parbox[t]{0.9\textwidth}{
  Expectation: $0$,\\ 
  Variation: $\sqrt{2\epsilon N ^{3}\left\lVert A _{k}^{\dagger }\mu \right\rVert_{2}^{2}\sigma ^{2}}\cdot \sqrt{2\epsilon N \ln \left(\frac{e}{2\epsilon }\right)}$,\\ 
  Techniques: Lemma \ref{lemma: combination number}, union bound over $S ^{\prime }$ and Chebyshev's inequality.
}
\npar

The sum of the expectations for \eqref{eq: decomposition of projection} is exactly $k \left|S ^{\prime }\right|\sigma ^{2}$, which is the second term in the LHS. It remains is to require all the variations above to be asymptotically less than the term II, which is equivalent to the condition $\left\lVert A ^{\dagger }_{k}\mu \right\rVert_{2}^{}\gtrsim E _{\text{raw}}(\epsilon ,K,N,\sigma ^{2},k)$ with simple algebra.

Now consider the general corrupted observations $\mathbf{Y}$. From the arguments above, we know that the original samples $\tilde{\mathbf{Y}}$ is consistent with $\phi _{e}$ with high (constant) probability given the condition \eqref{eq: main upper bound 1}. Combining with the simple fact that if $\tilde{\mathbf{Y}}$ is consistent, then $\tilde{\mathbf{Y}}_{S}$ is consistent for any $S \subset [N]$ with $\left|S\right|\ge (1-\epsilon )N$, we have, therefore, with high probability, there exists at least one consistent subset within $\mathbf{Y}$ (at least $\mathbf{Y}_{[N]\backslash C}=\tilde{\mathbf{Y}}_{[N]\backslash C}$) for $\phi _{e}$. The theoretical algorithm works by scanning through all the subsets of $\mathbf{Y}$ with cardinality no less than $(1-\epsilon )N$ to check their consistency. Denote any found consistent subset as $Y _{S _{0}}$. Then the testing result on $\tilde{\mathbf{Y}}$ is recovered through the connection that $\phi _{e}(\mathbf{Y}_{S _{0}})=\phi _{e}(\mathbf{Y}_{S _{0} \cap [N]\backslash C})=\phi _{e}(\mathbf{Y}_{[N]\backslash C})=\phi _{e}(\tilde{\mathbf{Y}})$ and the fact that $\left[S _{0}\cap [N]\backslash C\right]\ge (1-2\epsilon )N$. Since $\phi _{e}$ itself is a valid test for the uncontaminated testing problem, the proof is completed.
\end{proof}

\subsection{Necessary Additional Concentration Inequalities for Theorem \ref{theorem: raw main upper bound 2}}\label{subsection: necessary concentration inequality for the main upper bound 2}
\begin{lemma}[]\label{lemma: bounds on the l2-operator norm of X^TX and XX^T}
    Assume $\mathbf{X} \in \mathbb{R}^{n \times d}$ and each row is independently drawn from $\mathcal{N}(\mathbf{0},\mathbf{\Sigma })$ with $\mathbf{0}\preceq \mathbf{\Sigma }\preceq \mathbf{I}_{d}$ and $\text{tr}(\mathbf{\Sigma })=k \le d$, then for $\delta >0$, with probability higher than $1-2\delta $, we have:\\ 
    \begin{tabular}{rl}
        (i), $n \ge k$: & $\left\lVert \mathbf{X} ^{\top }\mathbf{X}-n \mathbf{\Sigma }\right\rVert_{2}^{}\lesssim \sqrt{nk}+\sqrt{n \ln \left(\frac{1}{\delta }\right)}+\ln \left(\frac{1}{\delta }\right)$, \\ 
        (ii), $n <k$: & $\left\lVert \mathbf{X}\mathbf{X} ^{\top } -k \mathbf{I}_{n}\right\rVert_{2}^{}\lesssim \sqrt{nk}+\sqrt{k \ln \left(\frac{1}{\delta }\right)}+\ln \left(\frac{1}{\delta }\right)$.
    \end{tabular}
\end{lemma}
\begin{proof}
  For the concentration bound on the sample covariance matrix, the proof involves the Talagrand generic chaining method and Talagrand's majorizing measure theorem. We kindly refer the reader to Corollary 2 of \cite{9d45c567-9b1f-32ff-a930-b800c17fe77f} for such results and proof. Here we only prove (ii) via the Hanson--Wright inequality and standard $\epsilon $-net argument.

  From the definition of $\ell _{2}$ operator norm, we know 
  \begin{equation*}
    \left\lVert \mathbf{X}\mathbf{X}^{\top } -k \mathbf{I}_{n}\right\rVert_{2}^{}=\max\limits _{u \in \mathbb{S}^{n-1}}\left|\left\lVert \mathbf{X}^{\top } u\right\rVert_{2}^{2}-k\right|.
  \end{equation*}
  Since $\left\lVert u\right\rVert_{2}^{}=1$, we know $\mathbf{X}^{\top } u$ is also a Gaussian random variable with distribution $\mathcal{N}(\mathbf{0},\mathbf{\Sigma })$. Let $\mathbf{X}^{\top } =\sqrt{\mathbf{\Sigma }}\mathbf{Z}$ where $\mathbf{Z}\sim \mathcal{N}(\mathbf{0},\mathbf{I}_{d})$. For any fixed $u \in \mathbb{S}^{n-1}$, by the Hanson--Wright inequality, we have 
  \begin{equation*}
    \mathbb{P}\left(\left|\left\lVert \mathbf{X}^{\top } u\right\rVert_{2}^{2}-k\right|\ge t\right)=\mathbb{P}\left(\left|\mathbf{Z}^{\top } \mathbf{\Sigma }\mathbf{Z}-k\right|\ge t\right)\overset{\text{(i)}}{\le }2 \exp\left\{-c \min\limits \left\{\frac{t ^{2}}{k},t\right\}\right\},
  \end{equation*}
  where (i) is from the fact that $\mathbf{0}\preceq \mathbf{\Sigma }\preceq \mathbf{I}_{d}$ and therefore $\left\lVert \mathbf{\Sigma }\right\rVert_{F}^{2}\le \left\lVert \mathbf{\Sigma }\right\rVert_{F}^{}=k$, $\left\lVert \mathbf{\Sigma }\right\rVert_{2}^{}\le 1$.

  For $\mathbb{S}^{n-1}$, from standard $\epsilon $-net argument, we know that there exists a set $\mathcal{S}=\left\{v _{1},v _{2},\dots\right\}$ with $\left|\mathcal{S}\right|\le 9 ^{n}$ and $\left\lVert u-v _{i}\right\rVert_{2}^{}\le \frac{1}{4}$ for any $u \in \mathbb{S}^{n-1}$ and some $v _{i}\in \mathcal{S}$. Denote $Q=\mathbf{X}\mathbf{X}^{\top } -k \mathbf{I}_{n}$, then for any $u \in \mathbb{S}^{n-1}$ and $v _{u}\in \mathcal{S}$ such that $\left\lVert u-v _{u}\right\rVert_{2}^{}\le \frac{1}{4}$, we have 
  \begin{equation*}
    \begin{aligned}
      \left|u ^{\top } Qu-v _{u}^{\top } Qv _{u}\right|&=\left|u ^{\top } Q(u-v _{u})+(u-v _{u})^{\top } Qv _{u}\right|\\ 
      & \le \left|u ^{\top } Q(u-v _{u})\right|+\left|(u-v _{u})^{\top } Qv _{u}\right|\\ 
      &\le \left\lVert u\right\rVert_{2}^{}\left\lVert Q\right\rVert_{2}^{}\left\lVert u-v _{u}\right\rVert_{2}^{}+\left\lVert u-v _{u}\right\rVert_{2}^{}\left\lVert M\right\rVert_{2}^{}\left\lVert v _{u}\right\rVert_{2}^{}\\ 
      &\le \frac{1}{2}\left\lVert Q\right\rVert_{2}^{}.
    \end{aligned}
  \end{equation*}
  Therefore, we have 
  \begin{equation*}
    \left\lVert u ^{\top } Qu\right\rVert_{2}^{}\le \left| v _{u}^{\top } Qv _{u}\right|_{2}^{}+\frac{1}{2}\left\lVert Q\right\rVert_{2}^{}\le \sup\limits_{v \in \mathcal{S}}\left|v ^{\top } Qv\right|+\frac{1}{2}\left\lVert Q\right\rVert_{2}^{}.
  \end{equation*}
  Taking supreme over $u$, we have 
  \begin{equation*}
    \left\lVert Q\right\rVert_{2}^{}\le 2 \sup\limits_{v \in \mathcal{S}}\left|v ^{\top } Qv\right|.
  \end{equation*}
  Finally, we have 
  \begin{equation*}
    \mathbb{P}\left(\left\lVert Q\right\rVert_{2}^{}\ge t\right)\le \mathbb{P}\left(\sup\limits_{v \in \mathcal{S}}\left|v ^{\top } Qv\right|\ge \frac{t}{2}\right)\le 9 ^{n}\cdot 2 \exp\left\{-c \min\limits \left\{\frac{t ^{2}}{k},t\right\}\right\}.
  \end{equation*}
  Simple calculations and the condition that $n<k$ yield that with probability greater than $1-2\delta $ we have 
  \begin{equation*}
    \left\lVert Q\right\rVert_{2}^{}\lesssim \sqrt{kn}+\sqrt{k \ln \left(\frac{1}{\delta }\right)}+\ln \left(\frac{1}{\delta }\right).
  \end{equation*}
  The proof is completed.
\end{proof}

Distinct from Lemma \ref{lemma: bounds on the l2-operator norm of X^TX and XX^T}, Lemma \ref{lemma: bound on the max eigenvalue of the covariance matrix} below bounds the $\ell _{2}$-operator norm of the weighted empirical covariance $\sum\limits_{i=1}^{n}\omega _{i}X _{i}X _{i}^{\top }$ under small weights $\omega $. This result plays a crucial role in verifying the second condition of \hyperref[def_omega_regularity_init]{$\omega $-regularity} subsequent to Algorithm \ref{algorithm n > k} (or \ref{algorithm n <= k}). While the proof builds upon Fact 4.2 in \cite{dong2019quantumentropyscoringfast}, our analysis is carefully tailored to handle the specific structural assumptions imposed on $\mathbf{\Sigma }$.

\begin{lemma}[]\label{lemma: bound on the max eigenvalue of the covariance matrix}
    Assume $\mathbf{X} \in \mathbb{R}^{n \times d}$ and each row is independently drawn from $\mathcal{N}(0,\mathbf{\Sigma })$ where $\mathbf{0}\preceq \mathbf{\Sigma }\preceq \mathbf{I}_{d}$ and $\text{tr}(\mathbf{\Sigma })=k$. Then, with some constants $c _{1},c _{2}>0$, the following holds with probability at least $1-\delta $ for any $\omega $ that $0 \le \omega _{i}\le 1$ and $\left\lVert \omega \right\rVert_{1}^{}\le c _{2}\cdot \epsilon n$:
    \begin{equation*}
        \left\lVert \sum\limits_{i=1}^{n}\omega _{i}X _{i}X _{i}^{\top } \right\rVert_{2}^{}\le c _{1}\left(\epsilon n \ln \left(\frac{1}{\epsilon }\right)+k+\ln \left(\frac{1}{\delta }\right)\right).
    \end{equation*}
\end{lemma}

\begin{proof}
  Without loss of generality, we assume $c _{2}=1$. For any fixed $v \in \mathbb{S}^{d-1}$, We have:
  \begin{equation*}
    v ^{\top } \left(\sum\limits _{i=1}^{n}\omega _{i}X _{i}X _{i}^{\top }\right)v=\sum\limits _{i=1}^{n}\omega _{i}(v ^{\top } X _{i})^{2}.
  \end{equation*}
  Since $(v ^{\top } X _{i})^{2}\ge 0$, we further assume $\left\lVert \omega \right\rVert_{1}^{}=\epsilon n$. The set $\left\{\omega \left\lvert\right. \left\lVert \omega \right\rVert_{1}^{}=\epsilon n\right\}$ is convex, consequently the RHS above is maximized at some vertex $\omega _{S}:=\mathbf{1}_{S},\left|S\right|=\epsilon n$. Therefore, we have 
  \begin{equation*}
    \sum\limits _{i=1}^{n}\omega _{i}(v ^{\top } X _{i})^{2}=\sum\limits _{i \in S}^{}\left[(v ^{\top } X _{i})^{2}-v ^{\top } \mathbf{\Sigma }v\right]+\epsilon n(v ^{\top } \mathbf{\Sigma }v)\le \left\lVert \sum\limits _{i \in S}^{}X _{i}X _{i}^{\top } -\epsilon n \mathbf{\Sigma }\right\rVert_{2}^{}+\epsilon n.
  \end{equation*}
  From Corollary 2 of \cite{9d45c567-9b1f-32ff-a930-b800c17fe77f} (also see Lemma \ref{lemma: bounds on the l2-operator norm of X^TX and XX^T}), for any fixed $S$ with $\left|S\right|=\epsilon n$, we have 
  \begin{equation*}
    \left\lVert \sum\limits _{i \in S}^{}X _{i}X _{i}^{\top } -\epsilon n \mathbf{\Sigma }\right\rVert_{2}^{}\lesssim \sqrt{\epsilon nk}+k+\sqrt{\epsilon n t}+t
  \end{equation*}
  with probability greater than $1-e ^{-t}$. From Lemma \ref{lemma: combination number}, the total amount of such $S$ is at most $\left(\frac{e}{\epsilon }\right)^{\epsilon n}$. Setting $t=\delta \cdot \left(\frac{\epsilon }{e}\right)^{\epsilon n}$ in the inequality above and using the union bound, we have:
  \begin{equation*}
    \mathbb{P}\left(\left\lVert \sum\limits _{i \in S}^{}X _{i}X _{i}^{\top } -\epsilon n \mathbf{\Sigma }\right\rVert_{2}^{}\lesssim \sqrt{\epsilon nk}+k+\sqrt{\epsilon n \ln \left(\frac{1}{\delta }\right)}+\ln \left(\frac{1}{\delta }\right)+\epsilon n \ln \left(\frac{1}{\epsilon }\right), \forall S, \left|S\right|=\epsilon n\right)\ge 1-\delta .
  \end{equation*}
  Finally, we have
  \begin{equation*}
    \sqrt{\epsilon nk}+k+\sqrt{\epsilon n \ln \left(\frac{1}{\delta }\right)}+\ln \left(\frac{1}{\delta }\right)+\epsilon n \ln \left(\frac{1}{\epsilon }\right)\lesssim \epsilon n \ln \left(\frac{1}{\epsilon }\right)+k+\ln \left(\frac{1}{\delta }\right).
  \end{equation*}
  The proof is completed by taking the supreme over $v \in \mathbb{S}^{d-1}$.
\end{proof}

\end{document}